\documentclass[reqno,final]{amsart}

\usepackage[natbib=true,hyperref=true,backref=true,maxnames=2,backend=biber]{biblatex}
\makeatletter
\DeclareFieldFormat{postnote}{\blx@mkpageprefix{section}{#1}}
\makeatother
\usepackage[breaklinks,bookmarks,final]{hyperref}
\usepackage{latexsym}
\usepackage{mathrsfs}
\usepackage{amssymb}
\usepackage{euscript}
\usepackage[mathbf]{euler}
\usepackage{mathtools}
\usepackage{mathabx}
\usepackage{cmap}
\usepackage{paresse}
\usepackage[all]{onlyamsmath}
\usepackage{ifthen}
\newcommand{\Ifempty}[3]{\ifthenelse{\equal{#1}{}}{#2}{#3}}

\usepackage[all]{xy}
\usepackage{verbatim}

\newcommand{\DD}{\mathbb{D}}

\newcommand{\Cc}{\mathcal{C}}
\newcommand{\Dd}{\mathcal{D}}

\newcommand{\rA}{\mathrm{A}}

\newcommand{\rM}{\mathrm{M}}

\newcommand{\eps}{\epsilon}

\newcommand{\w}{\omega}

\newcommand{\cat}[1]{\mathcal{#1}} % category
\newcommand{\ftr}[1]{\mathscr{#1}} % functor
\newcommand{\nat}[1]{§ #1} % natural map
\newcommand{\obj}[1]{\mathbf{#1}} % object, definable set
\renewcommand{\th}[1]{\cat{#1}} % theory
\newcommand{\md}[1]{\ftr{#1}} % model
\newcommand{\dcs}[1]{\md{#1}} % definably closed set
\newcommand{\sch}[1]{\ftr{#1}} % scheme
\renewcommand{\rM}{\md{M}}

\renewcommand{\Cc}{\cat{C}}
\renewcommand{\Dd}{\cat{D}}
\renewcommand{\rA}{\dcs{A}}
\renewcommand{\DD}{\mathcal{P}}
\newcommand{\oA}{\obj{A}}
\newcommand{\oB}{\obj{B}}
\newcommand{\oC}{\obj{C}}
\newcommand{\oD}{\obj{D}}
\newcommand{\oE}{\obj{E}}
\newcommand{\oU}{\obj{U}}

\newcommand{\oW}{\obj{W}}
\newcommand{\oX}{\obj{X}}
\newcommand{\oY}{\obj{Y}}
\newcommand{\oZ}{\obj{Z}}
\newcommand{\oQ}{\obj{Q}}
\newcommand{\oG}{\obj{G}}
\newcommand{\oH}{\obj{H}}
\newcommand{\fF}{\ftr{F}}
\newcommand{\LL}{\obj{L}}
\newcommand{\tT}{\th{T}}

\newcommand{\ul}[1]{\underline{#1}}

\newcommand{\ti}[1]{\widetilde{#1}}
\newcommand{\ra}[1][]{\xrightarrow{#1}}
\newcommand{\mt}{\mapsto}
\newcommand{\wbar}{\widebar}
\newcommand{\Ten}{\otimes}
\newcommand{\ten}{\hat{\Ten}}
\newcommand{\YS}{\star}
\newcommand{\Id}{{Id}}
\newcommand{\Ob}[1]{{0\ra{#1}_0\ra[i_{#1}]{#1}_1\ra[\pi_{#1}]{#1}_0\ra 0}}
\newcommand{\dd}{\mathfrak{D}}

\providecommand{\1}{\mathbf{1}}
\renewcommand{\1}{\mathbf{1}}
\providecommand{\2}{\mathbf{2}}
\renewcommand{\2}{\mathbf{2}}
\newcommand{\x}{\times}
\newcommand{\Isom}{\ra[\sim]}
\newcommand{\kk}{\Bbbk}

\newcommand{\Hom}{\ul{Hom}}
\newcommand{\Aut}{\ul{Aut}}

\newcommand{\Co}[1]{\widecheck{{#1}}}

\newcommand{\Def}[1]{\emph{#1}}

\DeclareMathAlphabet{\mathpzc}{OT1}{pzc}{m}{it}

\renewcommand{\Vec}{\mathcal{V}ec}
\newcommand{\Rep}{\mathcal{R}ep}

\DeclareMathOperator{\rk}{rk}
\DeclareMathOperator{\dlog}{dlog}
\DeclareMathOperator{\im}{im}
\DeclareMathOperator{\Tr}{Tr}
\DeclareMathOperator{\Sch}{\mathbb{A}}
\renewcommand{\d}{\partial}

\newcommand{\Ind}[2][]{\operatorname{Ind}_{#1}({#2})}

\newcommand{\dcl}[1]{\operatorname{dcl}(#1)}

\newcommand{\point}[1][]{\subsection{#1}}
\newcommand{\subpoint}[1][]{\subsubsection{#1}}

\newcommand{\newthm}[2]{\newtheorem{#1}[subsection]{#2}\newtheorem{#1s}[subsubsection]{#2}}
\newthm{theorem}{Theorem}
\newthm{claim}{Claim}
\newthm{cor}{Corollary}
\newthm{prop}{Proposition}
\newthm{lemma}{Lemma}
\theoremstyle{remark}
\newthm{remark}{Remark}
\newthm{example}{Example}
\theoremstyle{definition}
\newthm{defn}{Definition}

\newcommand{\DefAlias}[2]{\expandafter\xdef\csname #1\endcsname{#2}}
\newcommand{\CiteAlias}[2]{\DefAlias{CITE#1}{#2}}
\providecommand{\Cite}[2][]{}
\renewcommand{\Cite}[2][]{\Ifempty{#1}{\citet{\csname CITE#2\endcsname}}{\citet[#1]{\csname CITE#2\endcsname}}}

\CiteAlias{Dgal}{MR1972449}
\CiteAlias{CWM}{MR1712872}
\CiteAlias{LNM900II}{LNM900II}
\CiteAlias{Topos}{MR1300636}
\CiteAlias{prodef}{prodef}
\CiteAlias{Rep2Dim}{MR2100671}
\CiteAlias{EtHom}{MR883959}
\CiteAlias{SGA4}{MR0354652}
\CiteAlias{SGA1}{MR2017446}
\CiteAlias{groupoids}{groupoid}
\CiteAlias{sacks}{MR0398817}
\CiteAlias{etale}{MR559531}
\CiteAlias{Deligne}{MR1106898}
\CiteAlias{acfgroups}{MR1827833}
\CiteAlias{defaut}{defaut}
\CiteAlias{Makkai}{MR0505486}
\CiteAlias{cassidy}{MR0360611}
\CiteAlias{cassidy2}{MR0409426}
\CiteAlias{freyd}{MR0166240}
\CiteAlias{ovchinnikov2}{AR0703422}
\CiteAlias{ovchinnikov1}{AR0702846}
\CiteAlias{Saavedra}{MR0338002}
\CiteAlias{DCF}{MR1773702}
\CiteAlias{DCFv}{MR2215060}
\CiteAlias{Waterhouse}{MR547117}
\CiteAlias{Kolchin}{MR0568864}
\CiteAlias{Zilberint}{MR606796}
\CiteAlias{PoizatGalois}{MR727805}
\CiteAlias{HrStable}{MR1081816}
\CiteAlias{simple}{MR2140036}
\CiteAlias{Pillay}{MR2094550}
\CiteAlias{Obch}{MR2480859}
\CiteAlias{Obchi}{MR2426137}
\CiteAlias{marker}{MR1924282}
\CiteAlias{pillayMT}{PillayMT}
\CiteAlias{pillayDAG}{MR1452531}

\title{Model theory and the Tannakian formalism}
\author[M. Kamensky]{Moshe Kamensky}
\address{
  Department of Math \\
  University of Notre-Dame \\
  Notre-Dame, IN, 46556\\
  USA
}
\email{\url{mailto:kamensky.1@nd.edu}}
\urladdr{\url{http://mkamensky.notlong.com}}

\begin{document}
\begin{abstract}
  We draw the connection between the model theoretic notions of internality 
  and the binding group on one hand, and the Tannakian formalism on the 
  other.
  
  More precisely, we deduce the fundamental results of the Tannakian 
  formalism by associating to a Tannakian category a first order theory, and 
  applying the results on internality there. We also formulate the notion of 
  a differential tensor category, and a version of the Tannakian formalism 
  for differential linear groups, and show how the same techniques can be 
  used to deduce the analogous results in that context.
\end{abstract}
\maketitle

\section*{Introduction}
The aim of this paper is to exhibit the analogy and relationship between two 
seemingly unrelated theories. On the one hand, the Tannakian formalism, 
giving a duality theory between affine group schemes (or, more generally, 
gerbs) and a certain type of categories with additional structure, the 
Tannakian categories. On the other hand, a general notion of internality in 
model theory, valid for an arbitrary first order theory, that gives rise to a 
definable Galois group. The analogy is made precise by deriving (a weak 
version of) the fundamental theorem of the Tannakian duality 
(Theorem~\ref{tan:main}) using the model theoretic internality.

The Tannakian formalism assigns to a group \(\sch{G}\) over a field \(\kk\), 
its category of representations \(\Rep_\sch{G}\). In the version we are 
mainly interested in, due to Saavedra (\Cite{Saavedra}), the group is an 
affine group scheme over a field. A similar approach works with groups in 
other categories, the first due to Krein and Tannaka, concerned with locally 
compact topological groups.  Another example is provided in 
Section~\ref{dif}.  In the algebraic case, the category is the category of 
algebraic finite-dimensional representations.  This is a \(\kk\)-linear 
category, but the category structure alone is not sufficient to recover the 
group. One therefore considers the additional structure given by the tensor 
product. The Tannakian formalism says that \(\sch{G}\) can be recovered from 
this structure, together with the forgetful functor to the category of 
vector-spaces. The other half of the theory is a description of the tensor 
categories that arise as categories of representations: any tensor category 
satisfying suitable axioms is of the form \(\Rep_\sch{G}\), provided it has a 
``fibre functor'' into the category of vector spaces. Our main references for 
this subject are the first three sections of~\Cite{LNM900II} 
and~\Cite{Deligne}.

In model theory, internality was discovered by Zilber as a tool to study the 
structure of strongly minimal structure (\Cite{Zilberint}). Later, Poizat 
realised (in~\Cite{PoizatGalois}) that this notion can be used to treat the 
Galois theory of differential equations. The definable Galois correspondence 
outlined in Theorem~\ref{mod:internality} has its origins there. Later, the 
theory was generalised to larger classes of theories (\Cite{HrStable}, 
\Cite{simple}, etc.), and applied in various contexts (e.g., \Cite{Pillay} 
extended the differential Galois theory to arbitrary ``\(D\)-groups'' 
definable in \(DCF\)).

In appendix B of~\Cite{Dgal}, internality was reformulated in a way that 
holds in an arbitrary theory. One is interested in the group of automorphisms 
\(G\) of a definable set \(\oQ\) over another definable set \(\oC\).  A set 
\(\oQ\) is internal to another set \(\oC\) if, after extending the base 
parameters, any element of \(\oQ\) is definable over the elements of \(\oC\).  
The idea is that under this condition, \(\oQ\) is close enough to \(\oC\) so 
that the \(G\) has a chance to be definable, but the requirement that a base 
extension is required prevents it from being trivial.  The theorem is that 
indeed \(G\) is the group of points of a (pro-) definable group 
(see Theorem~\ref{mod:internality}). This theory is reformulated again 
in~\Cite{groupoids}, where the same construction is described as an abstract 
duality theory between definable groupoids in a theory \(\th{T}\), and 
certain expansions of it, called internal covers. It is this formulation that 
we use.

The main results of the paper appear in Sections~\ref{alg} and~\ref{dif}.  
In Section~\ref{alg} we apply internality to prove the fundamental result on 
Tannakian categories (Theorem~\ref{tan:main}). This is done by constructing, 
for a tensor category \(\Cc\) an internal cover \(T_\Cc\) of \(ACF_\kk\).  
Models of this theory correspond, roughly, with fibre functors on \(\Cc\).  
The theory of internality provides a definable group in \(ACF_\kk\), and this 
is the group corresponding to \(\Cc\). The other parts of the theory follow 
from the Galois theory, and from the abstract duality theory 
of~\Cite{groupoids}.  

The result we prove is weaker than the original Theorem~\ref{tan:main} in the 
following ways. First,~\ref{tan:main}\eqref{tan:mainexr} states that a 
certain functor is representable by an affine algebraic group, but we only 
prove that its restriction to the category of fields agrees with such a group 
(see also Question~\ref{q:functor} below). Second, our proof works only in 
characteristic \(0\).  Finally,~\ref{tan:main}\eqref{tan:equiv} is covered, 
in part, by the general model theoretic statement~\ref{mod:action}, but the 
rest of the proof is only sketched, since it is not significantly different 
from the proof found in~\Cite{LNM900II}.

On the other hand, the proof has the advantage that it is simple and more 
``geometric'' than the algebraic one. It also has the advantage that the 
method is applicable in a more general context. One such application, 
concerning usual Galois theory, is briefly discussed in 
Remark~\ref{alg:galois}.  A more detailed application appears in 
Section~\ref{dif}, where we define the notion of a differential tensor 
category, and explain how the same model theoretic approach gives an 
analogous theorem (Theorem~\ref{dif:main}) there (though the method again 
gives a weaker result, in the same way as in the algebraic case).  A similar 
result, using a somewhat different language, was first obtain using algebraic 
methods in~\citet{\CITEObchi,\CITEObch}. It seems obvious that similar 
formalisms are possible in other contexts (e.g., difference fields, real 
closed fields). We note that, though model theory provides a general method 
of proof of such results, the right notion of a ``tensor category'' is new, 
and cannot be deduced directly from this method.

The first two sections contain background material. Section~\ref{mod} 
contains some general notions from model theory, as well as a statement of 
internality result, and some auxiliary remarks on it. Section~\ref{tan} 
gives a short review of Tannakian categories, with enough terminology and 
results to state the main theorem. Both sections are provided with the hope 
that the paper will be accessible to a wide audience (both within model 
theory and outside it), and that the information contained there will provide 
the reader with enough information to at least get a feeling for the nature 
of the results.

We mention also that there is a result in the other direction: It is possible 
to define the analogue of internal covers and recover a group object in a 
general categorical context, and recover the model theoretic result from it.  
This result will appear separately.

\point[Questions]
Several questions (some vague) remain unanswered in the paper. The main ones, 
from my point of view, are the following.

\subpoint
As described in Remarks~\ref{alg:deligne1} and~\ref{alg:deligne2}, the main 
results of~\Cite{Deligne} have very natural model theoretic translations. It 
seems reasonable to expect that there is a model theoretic proof, especially 
of the main theorem (which translates to having no new structure on 
\(ACF_\kk\)), but I could not find it. The case when \(\Cc\) is neutral 
appears to be easier from the algebraic point of view (and is proven already 
in~\Cite{LNM900II}), but model theoretically I don't know how to do even this 
case.

\subpoint
Conversely, the results and methods of~\Cite{LNM900II} and~\Cite{Deligne} 
seem to suggest that one can do ``model theory'' inside a Tannakian category 
(and perhaps more generally in a tensor category). For example, as can be 
seen from~\ref{alg:equiv}, the statement of~\ref{tan:main}\eqref{tan:equiv} 
can be viewed as saying that \(\Cc\) has elimination of quantifiers and 
elimination of imaginaries. I think it would be interesting to find out if 
this makes sense, and make it precise.

\subpoint\label{q:functor}
What is a general model theoretic machinery to prove results as 
in~\ref{alg:mainexr}, but in full generality? (rather than just for fields).  
This can also be asked purely in terms of functors (though in this context 
there is probably no answer): Assume there is a functor \(\sch{F}\) from 
\(\kk\)-algebras to sets (or groups), and a scheme (or group scheme) 
\(\sch{G}\) with a map \(\sch{G}\ra\sch{F}\) that is a bijection on fields.  
What are conditions on \(\sch{F}\) that allow us to deduce that this map is 
an isomorphism?

\subsection*{Acknowledgements}
I would like to thank Ehud Hrushovski for suggesting this approach, and for 
help along the way. I would also like to thank Anand Pillay for his interest 
and useful remarks.

\section{Model theory and internality}\label{mod}
In this section, we recall some notions from model theory. Some elements of  
the model theoretic language are briefly sketched in~\ref{mod:intro} 
through~\ref{mod:imag}. These appear in most texts on model theory, for 
example \Cite{marker} or~\Cite{pillayMT}.

In the rest of the section, we recall the main model theoretic tool that we 
use, namely internality, and the definable Galois group. We shall, for the 
most part, follow the presentation and terminology from~\Cite{groupoids}.

The only new results in this section appear after the statement of the 
theorem (Theorem~\ref{mod:internality}), though most of them are implicit 
in~\Cite{groupoids}.

\point[Basic terminology]\label{mod:intro}
We briefly recall that a \Def{theory} is collection of statements (axioms) 
written in a fixed formal language, and a \Def{model} of the theory 
\(\th{T}\) is a structure consisting of an interpretation of the symbols in 
the language of \(\th{T}\), in which all the axioms of \(\th{T}\) hold.

For example, the theory \(\th{T}=ACF\) of algebraically closed fields can be 
written in a language containing symbols \((0,1,+,-,\cdot)\), and a model of 
this theory is a particular algebraically closed field.

A \Def{formula} is written in the same formal language, but has \Def{free 
variables}, into which elements of the model can be plugged. For the example 
of \(ACF\) above, any finite collection of polynomial equations and 
inequalities can be viewed as a formula, but there are other formulas, 
involving quantifiers. Any such formula \(\phi(x_1,\dots,x_n)\) (where 
\(x_1,\dots,x_n\) contain all free variables of \(\phi\)) thus determines a 
subset \(\phi(\md{M})\) of \(\md{M}^n\), for any model \(\md{M}\), namely, 
the set of all tuples \(\bar{a}\) for which \(\phi(\bar{a})\) holds. Two 
formulas \(\phi\) and \(\psi\) are \Def{equivalent} if 
\(\phi(\md{M})=\psi(\md{M})\) for all models \(\md{M}\). An equivalence class 
under this relation is called a \Def{definable set} (So, in this paper, 
``definable'' is without parameters.)

If \(\md{M}\) is the model of some theory, the set of all statements (in the 
underlying language) that are true in \(\md{M}\) is a theory 
\(\th{T}(\md{M})\), and \(\md{M}\) is a model of \(\th{T}(\md{M})\).

We will assume all our theories to be multi-sorted, i.e., the variables of a 
formula can take values in any number of disjoint sets. In fact, if 
\(\oX\) and \(\oY\) are sorts, we view \(\oX\x\oY\) as a new 
sort. In particular, any definable set is a subset of some sort. By a 
statement such as ``\(a\in\md{M}\)'' we will mean that \(a\) is an element of 
one of the sorts, interpreted in \(\md{M}\).

\point[Definable closure and automorphisms]
A \Def{definable function} \(f\) from a definable set \(\oX\) to another 
definable set \(\oY\) is a definable subset of \(\oX\x\oY\) that 
determines the graph of a function 
\(f_\md{M}:\oX(\md{M})\ra\oY(\md{M})\) for each model \(\md{M}\). If 
\(\md{M}\) is a model, \(A\subseteq\oX(\md{M})\) any subset, and 
\(b\in\oY(\md{M})\), then \(b\) is \Def{definable} over \(A\) if there is 
a definable function \(f:\oX^n\ra\oY\) for some \(n\), such that 
\(f(\bar{a})=b\) for some \(\bar{a}\in{}A^n\) (note that even though we are 
working in \(\md{M}\), it is \(\th{T}\) that should think that \(f\) is a 
function, rather than just \(\th{T}(\md{M})\)). A subset is \Def{definably 
closed} if it is closed under definable functions.  The  \Def{definable 
closure} \(\dcl{A}\) of a subset \(A\subseteq\md{M}\) is the smallest 
definably closed subset of \(\md{M}\) containing \(A\). We denote by 
\(\oY(A)\), the set of all elements in \(\oY(\md{M})\) definable over 
\(A\).

We note that if \(\th{T}\) is complete, then the definable closure 
\(\oX(0)\) of the empty set does not depend on the model. For general 
\(\th{T}\), we denote by \(\oX(0)\) the set of definable subsets of 
\(\oX\) containing one element.

More generally, a formula over \(A\) is a formula \(\phi(x,a)\), where 
\(\phi(x,y)\subseteq\oX\x\oY\) is a regular formula, and 
\(a\in\oY(A)\). It defines a subset of \(\oX(\md{M})\) for any model 
\(\md{M}\) containing \(A\), in the same way as regular formulas do.

An automorphism of a model \(\md{M}\) is a bijection \(\nat{f}\) from 
\(\md{M}\) to itself such that the induced bijection 
\(\nat{f}^n:\md{M}^n\ra\md{M}^n\) preserves any definable subset 
\(\oX(\md{M})\subseteq\md{M}^n\). If \(A\subseteq\md{M}\) is a set of 
parameters, the automorphism \(\nat{f}\) is \Def{over \(A\)} if it fixes all 
elements of \(A\). The group of all automorphisms of \(\md{M}\) over \(A\) is 
denoted by \(Aut(\md{M}/A)\). It is clear that \(\oX(A)\) is fixed pointwise 
by any automorphism over \(A\). The model \(\md{M}\) is called 
\Def{homogeneous} if, conversely, for any \(A\subseteq\md{M}\) of cardinality 
smaller than the cardinality of \(\md{M}\), for any definable set \(\oX\), 
any element of \(\oX(\md{M})\) fixed by all of \(Aut(\md{M}/A)\) is in 
\(\oX(A)\) (computed in \(\th{T}(\md{M})\)).  Homogeneous models are known to 
exist (cf.~\Cite[Ch.~20]{sacks}).

\point[Imaginaries and interpretations]\label{mod:imag}
A \Def{definable equivalence relation} on a definable set \(\oX\) is a 
definable subset of \(\oX\x\oX\) that determines an equivalence 
relation in any model. The theory \Def{eliminates imaginaries} if any 
equivalence relation has a quotient, i.e., any equivalence relation can be 
represented as \(f(x)=f(y)\) for some definable function \(f\) on 
\(\oX\).

If \(\th{T}\) eliminates imaginaries, an \Def{interpretation} \(i\) of 
another theory \(\th{T}_0\) in \(\th{T}\) specified by giving, for each sort 
\(\oX_0\) of \(\th{T}_0\) a definable set \(\oX=i(\oX_0)\) of \(\th{T}\), and 
for each atomic relation \(\oY_0\subseteq\oX_0\) of \(\th{T}_0\) a definable 
subset \(\oY=i(\oY_0)\) of \(\oX\), such that for any model \(\md{M}\) of 
\(\th{T}\), the sets \(\oX(\md{M})\) form a model \(\md{M}_0\) of 
\(\th{T}_0\) (when interpreted as a \(\th{T}_0\) structure in the obvious 
way). It follows that any definable set \(\oZ_0\) of \(\th{T}_0\) determines 
a definable set \(i(\oZ_0)\) in \(\th{T}\), and \(\md{M}_0\) can be viewed as 
a subset of \(\md{M}\) (there is a universal interpretation of any theory 
\(\th{T}\) in a theory \(\th{T}^{eq}\) that eliminates imaginaries. An 
interpretation of \(\th{T}_0\) in a general theory \(\th{T}\) is then defined 
to be an interpretation of \(\th{T}_0\) in \(\th{T}^{eq}\) in the sense 
already defined.)

If \(\md{M}\) is a model of a theory \(\th{T}\), and \(A\subseteq\md{M}\), we 
denote by \(\th{T}_A\) the theory obtained by adding constants for \(A\) to 
the language, and the axioms satisfied for \(A\) in \(\md{M}\) (in 
particular, \(\th{T}_A\) is complete). There is an obvious interpretation of 
\(\th{T}\) in \(\th{T}_A\).  The model \(\md{M}\) is in a natural way a model 
of \(\th{T}_A\), and for any definable set \(\oX\) of \(\th{T}\), \(\oX(A)\) 
in the sense of \(\th{T}\) is identified with \(\oX(0)\) in \(\th{T}_A\).

\point[Internal covers]\label{mod:cover}
An interpretation \(i\) of \(\th{T}_0\) in \(\th{T}\) is \Def{stably 
embedded} if any subset of the sorts of \(\th{T}_0\) definable in \(\th{T}\) 
with parameters from \(\th{T}\) is also definable in \(\th{T}_0\) with 
parameters from \(\th{T}_0\). More precisely, given definable sets \(\oY_0\) 
in \(\tT_0\), and \(\oX\) and \(\oZ\subseteq\oX\x{}i(\oY_0)\) in \(\tT\), 
there are definable sets \(\oW_0\) and \(\oZ_0\subseteq\oW_0\x\oY_0\) in 
\(\tT_0\), such that \(\tT\) implies that for each \(x\in\oX\), 
\(\oZ_x=i(\oZ_0)_w\) for some \(w\in{}i(\oW_0)\).

\begin{defn}
  A stably embedded interpretation \(i\) of \(\th{T}_0\) in \(\th{T}\) is an 
  \Def{internal cover} (of \(\th{T}_0\)) if there is a stably embedded 
  interpretation \(\pi\) of \(\tT\) in \(\tT_0\), such that \(\pi\circ{}i\) 
  is the identity.

  The theory \(\tT\) is also called an internal cover of \(\tT_0\) (if \(i\) 
  is understood).
\end{defn}

If \(\oQ\) is a definable set in \(\tT\) (an internal cover of \(\tT_0\)), 
the identity map on \(\pi(\oQ)\) comes, by assumption, from some bijection 
from \(\oQ\) to \(i(\pi(\oQ))\) in \(\tT\), definable with parameters. Hence 
there is a set of parameters \(A\), such that 
\(\dcl{\md{M}_0\cup{}A}=\md{M}\) for any model \(\md{M}\) of \(\tT\). Such a 
set \(A\) will be called a set of \Def{internality parameters}, and will 
always be taken to be definably closed. We denote by \(A_0\) the restriction 
of \(A\) to \(\th{T}_0\). (In the language of~\Cite{groupoids}, this implies 
that the corresponding groupoid  in \(\th{T}_0\) is equivalent to a groupoid 
with objects over \(A_0\).)

We note that by compactness, a finite number of elements of \(A\) suffices to 
define all elements of a given sort of \(\th{T}\) over \(\md{M}_0\). Thus, if 
\(\th{T}\) has a finite number of sorts that do not come from \(\th{T}_0\) 
over which everything is definable, then \(A\) can be taken to be generated 
by one tuple \(a\). In general, \(A\) can be thought of as an element of a 
pro-definable set.

The following formulation of the internality theorem is closest in language 
and strength to the one in~\Cite[prop.~1.5]{groupoids}. This is the language 
we would like to use later, but in the strength we actually require (namely, 
for \(\w\)-stable theories), it was already proved in~\Cite{PoizatGalois}.

\begin{theorem}[{\Cite[Appendix~B]{Dgal},\Cite[prop.~1.5]{groupoids}}]\label{mod:internality}
  Let \(\th{T}\) be an internal cover of \(\th{T}_0\). There is a pro-group 
  \(\oG\) in \(\th{T}\), together with a definable action 
  \(m_Q:\oG\x\oQ\ra\oQ\) of \(\oG\) on every definable set \(\oQ\) of 
  \(\th{T}\), such that for any model \(\md{M}\) of \(\th{T}\), 
  \(\oG(\md{M})\) is identified with \(Aut(\md{M}/\md{M}_0)\) through this 
  action (with \(\md{M}_0\) the restriction of \(\md{M}\) to \(\th{T}_0\)).

  Furthermore, given a set of internality parameters \(A\) (whose restriction 
  to \(\th{T}_0\) is \(A_0\)), there is a Galois correspondence between 
  \(A\)-definable pro-subgroups of \(\oG\) and definably closed subsets 
  \(A_0\subseteq{}B\subseteq{}A\). Such a subset \(B_H\) for a subgroup 
  \(\oH\) is always of the form \(\oC(A)\), where \(\oC\) is an 
  \(A\)-definable ind-set in \(\th{T}\), and \(\oH\) is the subgroup of 
  \(\oG\) fixing \(\oC(\md{M})\) pointwise. If \(\oH\) is normal, then 
  \(\oG/\oH(\md{M})\) is identified with \(Aut(\oC(\md{M})/\md{M}_0)\).
\end{theorem}

Here, a \Def{pro-group} is a filtering inverse system of definable groups.

\point
In fact, the result is slightly stronger. With notation as above, the 
assumption that \(\th{T}_0\) is stably embedded implies that any automorphism 
of \(\md{M}_0\) fixing \(\dcs{A}_0\) can be extended (uniquely) to an 
automorphism of \(\md{M}\) fixing \(\dcs{A}\).  In other words, we have a 
split exact sequence
\begin{equation}
  0\ra\oG(\md{M})\ra Aut(\md{M}/\dcs{A}_0)\ra Aut(\md{M}_0/\dcs{A}_0)\ra 0
\end{equation}
where \(\oG\) is as provided by Theorem~\ref{mod:internality}.

More generally, we have the following interpretation of \(\oG(\dcs{B})\).
\begin{prop}\label{mod:stronger}
  Let \(\dcs{B}_0\subseteq\md{M}_0\) be a definably closed set containing 
  \(A_0\), and let \(\dcs{B}=\dcl{A\cup\dcs{B}_0}\). Then 
  \(\oG(\dcs{B})=Aut(\dcs{B}/\dcs{B}_0)\).
\end{prop}
\begin{proof}
  Any \(g\in\oG(\dcs{B})\) acts as an automorphism, preserves \(\dcs{B}\) as 
  a set (since \(B\) is definably closed), and fixes \(B_0\) pointwise, so 
  \(\oG(\dcs{B})\subseteq{}Aut(\dcs{B}/\dcs{B}_0)\).
  
  Conversely, since \(\dcs{B}_0\) is definably closed and contains 
  \(\dcs{A}_0\), \(\dcs{B}\cap\md{M}=\dcs{B}_0\), and so any automorphism of 
  \(\dcs{B}\) over \(\dcs{B}_0\) extends to an automorphism of \(\dcs{B}\) 
  over \(\md{M}_0\).  Since \(\md{M}_0\) is stably embedded, this 
  automorphism extends to an automorphism of \(\md{M}\).  Thus, any element 
  of \(Aut(\dcs{B}/\dcs{B}_0)\) is represented by some \(g\in\oG(\md{M})\).  
  Since it is fixed by any automorphism fixing \(\dcs{B}\) pointwise, it is 
  in fact in \(\oG(\dcs{B})\).
\end{proof}

\point\label{mod:nogpd}
We make the assumption that a set \(\dcs{A}\) as above can be found (in some 
model) such that the corresponding set \(\dcs{A}_0\) is (the definable 
closure of) the empty set. This is not a real assumption in the current 
context, since the results will hold in general for the theory with 
parameters from \(\dcs{A}\).

Likewise, the interpretation of \(\th{T}_0\) in \(\th{T}\) factors through a 
maximal extension \(\th{T}_1\) of \(\th{T}_0\) (i.e., \(\th{T}_1\) is the 
theory, in \(\th{T}\), of the definable sets coming from \(\th{T}_0\)). We 
assume from now on that \(\th{T}_1=\th{T}_0\).

Combining the two assumptions, we get that 
\(\dcs{B}_0\mt\dcs{B}=\dcl{\dcs{A}\cup\dcs{B}_0}\) gives an equivalence 
between definably closed sets \(\dcs{B}_0\) of \(\th{T}_0\) and definably 
closed sets \(\dcs{B}\) containing \(\dcs{A}\) of \(\th{T}\). In particular, 
we have a definable group \(\oG_\dcs{A}\) in \(\th{T}_0\).  The following 
proposition says that all definable group action of \(\oG_\dcs{A}\) in 
\(\th{T}_0\) come from the canonical action of \(\oG\) in \(\th{T}\).

\begin{prop}\label{mod:action}
  Assume~\ref{mod:nogpd}, and let \(a:\oG_\dcs{A}\x\oD\ra\oD\) be a definable 
  group action in \(\th{T}_0\). There is a definable set \(\oX_\oD\) in 
  \(\th{T}\), and an \(\dcs{A}\)-definable isomorphism of \(\oG_\dcs{A}\) 
  actions from \(\oX_\oD\) to \(\oD\).  If \(\oD\) and \(\oE\) are two such 
  \(\oG_\dcs{A}\)-sets, and \(f:\oD\ra\oE\) is a definable map of 
  \(\oG_\dcs{A}\) sets in \(\th{T}_0\), it also comes from a definable map 
  \(F:\oX_\oD\ra\oX_\oE\).
\end{prop}
\begin{proof}
  The proof of Theorem~\ref{mod:internality} produces a definable 
  \(\oG\)-torsor \(\oX\) in \(\th{T}\), which is \(0\)-definable with our 
  assumptions.  There is an \(\dcs{A}\)-definable isomorphism 
  \(f_b:\oG\ra\oG_\dcs{A}\), which, composed with \(a\), defines an 
  \(\dcs{A}\)-definable action \(c_b:\oG\x\oD\ra\oD\). We set 
  \(\oX_\oD=\oX\x_\oG\oD\), i.e., \(\oX_\oD\) is the set of pairs 
  \((x,d)\in\oX\x\oD\), up to the equivalence \((gx,d)=(x,c_b(g,d))\) for 
  \(g\in\oG\). Since \(\oG_\dcs{A}\) is in \(\th{T}_0\), 
  \(f_{hb}(hgh^{-1})=f_b(g)\) for all \(h,g\in{}G\), so the equivalence 
  relation is invariant under the action of \(G\), and hence \(\oX_\oD\) is 
  definable without parameters.
  
  If \(b\in\oX(\dcs{A})\) is any element (known to exist), the map 
  \([b,d]\mt{}d\) is an isomorphism of \(\oG_\dcs{A}\) actions from 
  \(\oX_\oD\) to \(\oD\). The proof for maps is similar.
\end{proof}

\section{Tannakian categories}\label{tensor}\label{tan}
In this section we review, without proofs, the definitions and basic 
properties of Tannakian categories. The proofs, as well as more details 
(originally from~\Cite{Saavedra}) can be found in the first two sections 
of~\Cite{LNM900II}. The section contains no new results (though the 
terminology is slightly different, and follows in part~\Cite[ch.~VII]{CWM}).

Let \(\sch{G}\) be an affine algebraic group (or, more generally, an affine 
groups scheme) over a field \(\kk\). The category \(\Rep_\sch{G}\) of 
finite-dimensional representations of \(\sch{G}\) over \(\kk\), admits, in 
addition to the category structure, a tensor product operation on the 
objects. With this structure, and with the forgetful functor into the 
category of vector spaces over \(\kk\) it satisfies the axioms of a 
(neutralised) Tannakian category (see Example~\ref{tan:repG}). The main 
theorem of~\Cite{Saavedra} (Theorem~\ref{tan:main}) asserts that categories 
satisfying these axioms, summarised below, are in fact precisely of the form 
\(\Rep_\sch{G}\), and that furthermore, the group \(\sch{G}\) can be 
recovered from the (tensor) category structure.

\point[Monoidal categories]
Recall that a \Def{symmetric monoidal category} is given by a tuple 
\((\Cc,\Ten,\phi,\psi)\), where:
\begin{enumerate}
  \item \(\Cc\) is a category
  \item \(\Ten\) is a functor \(\Cc\x\Cc\ra\Cc\), 
    \((\oX,\oY)\mt\oX\Ten\oY\) (in other words, the operation 
    is functorial in each coordinate separately).
  \item \(\phi\) is a collection of functorial isomorphisms
    \begin{equation*}
      \phi_{\oX,\oY,\oZ}:(\oX\Ten\oY)\Ten\oZ\Isom
                                      \oX\Ten(\oY\Ten\oZ)
    \end{equation*}
    one for each triple \(\oX,\oY,\oZ\) of objects of 
    \(\Cc\).  It is called the \Def{associativity constraint}.
  \item \(\psi\) is a collection of functorial isomorphisms 
    \(\psi_{\oX,\oY}:\oX\Ten\oY\Isom\oY\Ten\oX\), 
    called the \Def{commutativity constraint}.
\end{enumerate}
The commutativity constraint is required to satisfy 
\(\psi_{\oX,\oY}\circ\psi_{\oY,\oX}=id_{\oX\Ten\oY}\) 
for all objects \(\oX,\oY\). In addition, \(\phi\) and \(\psi\) are 
required to satisfy certain ``pentagon'' and ``hexagon'' identities, which 
ensure that any two tensor expressions computed from the same set of objects 
are \emph{canonically} isomorphic.

Finally, \(\Cc\) is required have an \Def{identity object}: this is an 
object \(\1\), together with an isomorphism \(u:\1\ra\1\Ten\1\), such that 
\(\oX\mt\1\Ten\oX\) is an equivalence of categories. It follows 
(\Cite[prop.~1.3]{LNM900II}) that \((\1,u)\) is unique up to a unique 
isomorphism with the property that \(u\) can be uniquely extended to 
isomorphisms \(l_\oX:\oX\ra\1\Ten\oX\), commuting with \(\phi\) 
and \(\psi\).

We will usually drop \(\phi\) and \(\psi\) from the notation, and choose a 
particular identity object \(\1\). We will also drop the adjective 
``symmetric''.

\point[Rigidity]\label{tan:rigid}
A monoidal category \((\Cc,\Ten)\) is \Def{closed} if for any two objects 
\(\oX\) and \(\oY\), there is an object \(\Hom(\oX,\oY)\) and a functorial 
isomorphism \(Hom(\oZ\Ten\oX,\oY)\Isom{}Hom(\oZ,\Hom(\oX,\oY))\). In this 
case, the \Def{dual} \(\Co{\oX}\) of \(\oX\) is defined to be 
\(\Hom(\oX,\1)\). The symmetry constraint and the functoriality determine, 
for each object \(\oX\), a map \(\oX\ra\Co{\Co{\oX}}\) and for any four 
objects \(\oX_1\), \(\oX_2\), \(\oY_1\), \(\oY_2\), a map
\begin{equation*}
  \Hom(\oX_1,\oY_1)\Ten\Hom(\oX_2,\oY_2)\ra
  \Hom(\oX_1\Ten \oX_2,\oY_1\Ten \oY_2)
\end{equation*}

\(\Cc\) is said to be \Def{rigid} if it is closed, and all these maps are 
isomorphisms (The latter requirement can be viewed as a formal analogue of 
the objects being finite-dimensional.)

We note that setting \(\oY_1=\oX_2=\1\) in the above map, we get an 
isomorphism \(\Co{\oX}\Ten\oY\ra\Hom(\oX,\oY)\). In 
particular, a map \(f:\oX\ra\oY\) corresponds to a global section 
\(\1\ra\Hom(\oX,\oY)\), hence to a map 
\(\Co{f}:\1\ra\Co{\oX}\Ten\oY\).

\point[The rank of an object]
Let \(f:\oX\ra\oX\) be an endomorphism in a rigid category. As above, 
it corresponds to a morphism \(\Co{f}:\1\ra\Co{\oX}\Ten\oX\).  
Composing with the commutativity followed by the evaluation we get a morphism 
\(\Tr_\oX(f):\1\ra\1\) called the \Def{trace} of \(f\). The trace map 
\(\Tr_\oX:End(X)\ra\kk:=End(\1)\) is multiplicative (with respect to the 
tensor product), and \(\Tr_\1\) is the identity.  The \Def{rank} of an object 
\(\oX\) is defined to be \(\Tr_\oX(id_\oX)\). Thus it is an 
element of \(\kk\).

\point[Tensor categories]\label{tan:gen}
A \Def{tensor category} is a rigid monoidal category \(\Cc\) which is 
abelian, with \(\Ten\) additive in each coordinate. It follows that \(\Ten\) 
is exact in each coordinate, that \(\kk=End(\1)\) is a commutative ring, that 
\(\Cc\) has a natural \(\kk\)-linear structure, where \(\Ten\) is 
\(\kk\)-bilinear, and \(\Tr_\oX\) is \(\kk\)-linear.

Given an object \(\oX\) of a tensor category \(\Cc\), let  
\(\Cc_\oX\) be the full subcategory of \(\Cc\) whose objects are 
isomorphic to sub-quotients of finite sums of tensor powers of \(\oX\) 
(including \(\oX^{\Ten^{-1}}=\Co{\oX}\) and its powers). The object \(\oX\) 
is called a \Def{tensor generator} for \(\Cc\) if any object of \(\Cc\) is 
isomorphic to an object of \(\Cc_\oX\).  For most of what follows we will 
restrict our attention to categories that have a tensor generator. In 
general, \(\Cc\) is a filtered limit of categories of this form.

\point[Tensor functors]\label{tan:functors}
Let \(\Cc\) and \(\cat{D}\) be monoidal categories. A \Def{tensor 
functor} from \(\Cc\) to \(\cat{D}\) is a pair \((\ftr{F},\nat{c})\), 
where \(\ftr{F}:\Cc\ra\cat{D}\) is a functor, and \(\nat{c}\) is a 
collection of isomorphisms 
\(\nat{c}_{\oX,\oY}:\fF(\oX)\Ten\fF(\oY)\Isom\fF(\oX\Ten\oY)\), compatible 
with the constraints, such that \((\fF,\nat{c})\) takes identity objects to 
identity objects.

When \(\Cc\) and \(\Dd\) are abelian, we assume \(\fF\) to be additive. In 
this case, \(\fF\) makes \(\kk_\Dd\) into a \(\kk_\Cc\)-algebra, and \(\fF\) 
is automatically \(\kk_\Cc\)-linear, and in this sense preserves the trace. 
In particular, \(\rk(\oX)=\rk(\ftr{F}(\oX))\).

If \((\ftr{F},\nat{c})\) and \((\ftr{G},\nat{d})\) are two tensor functors 
from \(\Cc\) to \(\cat{D}\), a map from \((\ftr{F},\nat{c})\) to 
\((\ftr{G},\nat{d})\) is a map of functors that commutes with \(\nat{c}\) and 
\(\nat{d}\). We denote by \(Aut(\ftr{F})\) the group of automorphisms of 
\(\ftr{F}\) as a tensor functor.

\point[Tannakian categories]\label{tan:neutral}
A \Def{neutral Tannakian category} over a field \(\kk\) is a tensor category  
\(\Cc\) with \(End(\1)=\kk\), which admits an exact tensor functor \(\w\) 
into the category \(\Vec_\kk\) of finite-dimensional \(\kk\)-vector spaces 
(with the usual tensor structure). Such a functor is called a \Def{fibre 
functor}.  It is automatically faithful, and is said to neutralise \(\Cc\).

Given a fibre functor \(\w\) and a (commutative) \(\kk\)-algebra \(A\), the 
functor \(\w\Ten_\kk{}A\) (with 
\((\w\Ten_\kk{}A)(\oX)=\w(\oX)\Ten_\kk{}A\)) is again an exact tensor 
functor (into a tensor subcategory of the category \(\Vec_A\) of projective 
finitely presented \(A\)-modules). We thus get a functor 
\(A\mt{}Aut(\w\Ten_\kk{}A)\) from the category of \(\kk\)-algebras to groups.  
This functor is denoted by \(\Aut^\Ten_\kk(\w)\).

\begin{example}[representations of a group]\label{tan:repG}
If \(\sch{G}\) is an affine group scheme over \(\kk\), the category 
\(\Cc=\Rep_\sch{G}\) of (finite-dimensional) representations of \(\sch{G}\) 
over \(\kk\), with the usual tensor product (and constraints) forms a 
monoidal category. It is rigid, with \(\Hom(\oX,\oY)\) the space of all 
linear maps between \(\oX\) and \(\oY\), and is abelian, hence it is a tensor 
category.  The trace of a map coincides with the image in \(\kk\) of the 
usual trace. Finally, \(\Cc\) is neutralised by the forgetful functor \(\w\).

If \(A\) is a \(\kk\)-algebra, and \(g\in\sch{G}(A)\), the action of 
\(\sch{G}\) exhibits \(g\) as an automorphism of \(\w\Ten_\kk{}A\).  This 
determines a map (of functors) \(\sch{G}\ra\Aut^\Ten_\kk(\w)\). The main 
theorem states that this map is an isomorphism, and that this example is the 
most general one:
\end{example}

\begin{theorem}[Saavedra]\label{tan:main}
Let \(\Cc\) be a neutral Tannakian category, and let \(\w:\Cc\ra\Vec_\kk\) be 
a fibre functor.
\begin{enumerate}
  \item\label{tan:mainexr}
     The functor \(\Aut^\Ten_\kk(\w)\) is representable by an affine group 
     scheme \(\sch{G}\) over \(\kk\).

  \item\label{tan:equiv}
    The fibre functor \(\w\) factors through a tensor equivalence 
    \(\w:\Cc\ra\Rep_\sch{G}\).

  \item\label{tan:mainrec}
    If \(\Cc=\Rep_\sch{H}\) for some affine group scheme \(\sch{H}\), then 
    the natural map \(\sch{H}\ra\sch{G}\) given in Example~\ref{tan:repG} is 
    an isomorphism.
\end{enumerate}
\end{theorem}

In the following section we present a proof of a slightly weaker statement 
(as explained in the introduction), using the model theoretic tools of 
Section~\ref{mod}.

\section{The theory associated with a tensor category}\label{alg}
In this section, we associate a theory with any Tannakian category, and use 
Theorem~\ref{mod:internality} to prove the main Theorem~\ref{tan:main}, in 
the case that the base field has characteristic \(0\) (and with the other 
caveats mentioned in the introduction).

\point\label{alg:theory}
Let \((\Cc,\Ten,\phi,\psi)\) be a tensor category over a field \(\kk\) of 
characteristic \(0\), such that the rank of every object is a natural number.  
Let \(\dcs{K}\) be an extension field of \(\kk\).  The theories \(T_\Cc\) 
and \(\ti{T_\Cc}\) (depending on \(\dcs{K}\), which is omitted from the 
notation) are defined as follows:
\subpoint
\(T_\Cc\) has a sort \(\LL\), as well as a sort \(V_\oX\) for every 
object \(\oX\) of \(\Cc\), and a function symbol 
\(v_f:V_\oX\ra{}V_\oY\) for any morphism \(f:\oX\ra\oY\).  
There are a binary operation symbol \(+_\oX\) on each \(V_\oX\), and 
a function symbol \(\cdot_\oX:\LL\x{}V_\oX\ra{}V_\oX\) for each object 
\(\oX\). The sort \(\LL\) contains a constant symbol for each element of 
\(\dcs{K}\).

\subpoint
The theory says that \(\LL\) is an algebraically closed field, whose field 
operations are given by \(+_\1\) and \(\cdot_\1\). The restriction of the 
operation to the constant symbols is given by the field structure on \(K\).

It also says that \(+_\oX\) and \(\cdot_\oX\) determine a vector 
space structure over \(\LL\) on every \(V_\oX\), and \(V_\oX\) has 
dimension \(\rk(\oX)\) over \(\LL\).  Each \(v_f\) is an \(\LL\)-linear 
map.

\subpoint
The \(\kk\)-linear category structure is reflected in the theory: 
\(v_{id_\oX}\) is the identity map for each \(\oX\), \(v_f=0\) if and 
only if \(f=0\), \(v_{f\circ{}g}=v_f\circ{}v_g\), if \(a\in\kk\) is viewed as 
an element of \(End(\oX)\), then \(v_a\) is multiplication by 
\(a\in\dcs{K}\).  Also, \(v_{f+g}=v_f+v_g\), and \(v_f\) is injective or 
surjective if and only if \(f\) is.

\subpoint
For any two objects \(\oX\) and \(\oY\) there is a further function 
symbol 
\(b_{\oX,\oY}:V_\oX\x{}V_\oY\ra{}V_{\oX\Ten\oY}\).  
The theory says that \(b_{\oX,\oY}\) is bilinear, and the induced map 
\(V_\oX\Ten_\LL{}V_\oY\ra{}V_{\oX\Ten\oY}\) is an 
isomorphism. Here \(\Ten_\LL\) is the usual tensor product of vector spaces 
over \(\LL\) (in a given model). This statement is first order, since the 
spaces are finite-dimensional.

The theory also says that 
\(v_{\psi_{\oX,\oY}}(b_{\oX,\oY}(x,y))=b_{\oY,\oX}(y,x)\), and similarly for 
\(\phi\).

If \(u\in{}V_\oX\) and \(v\in{}V_\oY\) (either terms or elements in 
some model), we write \(u\Ten{}v\) for \(b_{\oX,\oY}(u,v)\).

\subpoint
The theory \(\ti{T_\Cc}\) is the expansion of \(T_\Cc\) by an extra 
sort \(P_\oX\), for every object \(\oX\), together with a surjective 
map \(\pi_\oX:V_\oX\ra{}P_\oX\) identifying \(P_\oX\) with 
the projective space associated with \(V_\oX\). It also includes function 
symbols \(p_f:P_\oX\ra{}P_\oY\) and 
\(d_{\oX,\oY}:P_\oX\x{}P_\oY\ra{}P_{\oX\Ten\oY}\) for 
any objects \(\oX,\oY\) and morphism \(f\) of \(\Cc\), and the 
theory says that they are the projectivisations of the corresponding maps 
\(v_f\) and \(b_{\oX,\oY}\).

\begin{prop}\label{alg:stable}
  The theory \(T_\Cc\) is stable. In particular, \(\LL\) is stably embedded.  
  In each model, \(\LL\) is a pure algebraically closed field (possibly with 
  additional constants).
\end{prop}
\begin{proof}
  A choice of basis for each vector space identifies it with a power of 
  \(\LL\). All linear and bilinear maps are definable in the pure field 
  structure on \(\LL\). Since \(\LL\) is stable and stability is not affected 
  by parameters, so is \(T_\Cc\).
\end{proof}

\begin{prop}\label{alg:eqei}
  The theories \(T_\Cc\) and \(\ti{T}_\Cc\) eliminate quantifiers (possibly 
  after naming some constants). The theory \(\ti{T}_\Cc\) eliminates 
  imaginaries.
\end{prop}
\begin{proof}
  That \(\ti{T}_\Cc\) eliminates imaginaries is precisely the statement 
  of~\Cite[prop.~4.2]{groupoids} (see also the proof of 
  Proposition~\ref{dif:ei}).  The proof of quantifier elimination is similar: 
  since all sorts are interpretable (with parameters) in \(ACF\), all 
  definable sets are boolean combinations of Zariski-closed subsets. A 
  Zariski-closed subset of \(V_\oX\) or \(P_\oX\) is the set of zeroes of 
  polynomials, which are elements of the symmetric algebra on \(\Co{V_\oX}\).  
  But \(\Co{V_\oX}\) is identified with \(V_{\Co{\oX}}\), and the tensor 
  algebra on it, as well as the action on \(V_\oX\) or \(P_\oX\) is definable 
  without quantifiers.
\end{proof}

\point[Models and fibre functors]\label{alg:modfib}
Models of \(T_\Cc\) are essentially fibre functors over algebraically 
closed fields. More precisely, we have the following statements.

\begin{props}\label{alg:modisfib}
  Let \(\rM\) be a model of \(T_\Cc\), and let \(\dcs{A}\) be a subset of 
  \(\rM\) such that for any object \(\oX\), \(V_\oX(\dcs{A})\) contains a 
  basis of \(V_\oX(\md{M})\) over \(\LL(\md{M})\). Then 
  \(\oX\mt{}V_\oX(\dcs{A})\) determines a fibre functor over the field 
  extension \(\LL(\dcs{A})\) of \(\dcs{K}\).
\end{props}
\begin{proof}
  We first note that since (by definition) \(\LL(\dcs{A})\) is definably 
  closed, it is indeed a field. Likewise, since each \(V_\oX(\dcs{A})\) 
  is definably closed, it is a vector space over \(\LL(\dcs{A})\), and the 
  \(V_\oX(\dcs{A})\) are closed under the linear maps that come from 
  \(\Cc\). Thus \(\oX\mt{}V_\oX(\dcs{A})\) defines a 
  \(\kk\)-linear functor from \(\Cc\) to the category of vector spaces 
  over \(\LL(\dcs{A})\).
  
  The map \(b_{\oX,\oY}\) determines a map \(c_{\oX,\oY}:
  V_\oX(\dcs{A})\Ten_{\LL(\dcs{A})}V_\oY(\dcs{A})\ra
  V_{\oX\Ten\oY}(\dcs{A})\).  Since the map 
  \(V_\oX(\rA)\Ten_{\LL(\rA)}V_\oY(\rA)\ra{}V_\oX(\rM)\Ten_{\LL(\rM)}V_\oY(\rM)\) 
  is injective, so is \(c_{\oX,\oY}\), and so the whole statement 
  will follow from knowing that the dimension of \(V_\oX(\rA)\) over 
  \(\LL(\rA)\) is \(\rk(\oX)\). Since \(V_\oX(\rA)\) contains a basis 
  over \(\LL(\rM)\), we know that it is at least \(\rk(\oX)\).

  For any \(a_1,\dots,a_n\in{}V_\oX(\rA)\), the set of all 
  \((x_1,\dots,x_n)\in\LL^n\) such that \(\sum{}x_ia_i=0\) is an 
  \(\rA\)-definable subspace of \(\LL^n\), Since \(\LL\) is stably embedded, 
  it is \(\LL(\rA)\) definable.  Hence it has an \(\LL(\rA)\)-definable basis 
  (\Cite[Lemma~4.10]{etale}). This shows that the dimension is correct.
\end{proof}

\begin{props}
  Let \(\rA\) be a subset of a model \(\rM\) of \(T_\Cc\) such that 
  \(\dcl{\rA\cup\LL(\rM)}=\rM\), and such that \(\LL(\rA)\) is algebraically 
  closed.  Then \(\dcl{\rA}\) is a model. In particular, 
  \(\oX\mt{}V_\oX(\rA)\) determines a fibre functor for \(\Cc\) over 
  \(\LL(\rA)\).
\end{props}
\begin{proof}
  We only need to prove that each \(V_\oX(\rA)\) contains a basis over 
  \(\LL(\rM)\).  By assumption, there is an \(\rA\)-definable map \(f\) from 
  a subset \(\oU\) of \(\LL^m\) to \(V_\oX^n\), where 
  \(n=\rk(\oX)\), such that \(f(u)\) is a basis of \(V_\oX(\rM)\) for 
  all \(u\in\oU(\rM)\). Since \(\oU\) is \(\LL(\rA)\)-definable, 
  and \(\LL(\rA)\) is algebraically closed, \(\oU(\rA)\) contains a point 
  \(u\), and \(f(u)\) is thus a basis in \(V_\oX(\rA)\).
\end{proof}

\begin{props}\label{alg:fibismod}
  Conversely, if \(\dcs{K}_1\) is a field extension of \(\dcs{K}\), 
  \(\w:\Cc\ra\Vec_{\dcs{K}_1}\) is a fibre functor, and \(\bar{\dcs{K}}\) is 
  an algebraically closed field containing \(\dcs{K}_1\), the assignment 
  \(V_\oX\mt\w(\oX)\Ten_{\dcs{K}_1}\bar{\dcs{K}}\) determines a model 
  \(\rM_\w\) of \(T_\Cc\), such that \(\rA=\w\subseteq\rM\) satisfies 
  \(\dcl{\rA\cup\LL(\rM_\w)}=\rM_\w\), and \(\LL(\rA)=\dcs{K}_1\).

  In particular, \(T_\Cc\) is consistent if and only if \(\Cc\) is has a 
  fibre functor.
\end{props}
\begin{proof}
  Since \(\bar{\w}=\w\Ten_{\dcs{K}_1}\bar{\dcs{K}}\) is itself a fibre 
  functor, it is clear that \(\rM_\w\) is a model (\(b_{\oX,\oY}\) is defined 
  to be the composition of the canonical pairing 
  \(\bar{\w}(\oX)\x\bar{\w}(\oY)\ra\bar{\w}(\oX)\Ten_{\bar{\dcs{K}}}\bar{\w}(\oY)\) 
  with the structure map 
  \(c_{\oX,\oY}:\bar{\w}(\oX)\Ten\bar\w(\oY)\ra\bar\w(\oX\Ten\oY)\)).  Also, 
  it is clear that \(\dcl{\rA\cup\LL(\rM_\w)}=\rM_\w\). To prove that 
  \(\LL(\rA)=\dcs{K}_1\), we note that any automorphism of \(\bar{\dcs{K}}\) 
  over \(\dcs{K}_1\) extends to an automorphism of the model by acting the 
  second term.
\end{proof}

\point
From now on we assume that \(\Cc\) has a fibre functor \(\w\), and denote  
the theory of the model \(\rM_\w\) constructed in 
Proposition~\ref{alg:fibismod} by \(T_\w\) (so \(T_\w\) is a completion of 
\(T_\Cc\)).

We also assume that the category \(\Cc\) has a tensor generator 
(\ref{tan:gen}).  The general results are obtained as a limit of this case, 
in the usual way.  (This is not really necessary, since, as mentioned 
in~\ref{mod:cover}, the theory works in general, but it makes it easier to 
apply standard model theoretic results about definable, rather than 
pro-definable set.) We fix a tensor generator \(\oQ\).

\begin{claim}
  The theory \(T_\Cc\) is an internal cover of \(\LL\), viewed as an 
  interpretation of \(ACF_\dcs{K}\).
\end{claim}
\begin{proof}
  \(\LL\) is stably embedded in \(T_\Cc\) since \(T_\Cc\) is stable.  
  Given a model \(\rM\) of \(T_\Cc\), any element of \(V_\oX(\rM)\) 
  is a linear combination of elements of a basis \(a\) for \(V_\oX(\rM)\) 
  over \(\LL(\rM)\), hence is definable over \(a\) and \(\LL(\rM)\). We note 
  that if \(Q\) is a tensor generator, then a basis for \(V_Q\) determines a 
  basis for any other \(V_\oX\) (This is clear for tensor products, 
  direct sums, duals and quotients. For a sub-object \(\oY\ra\oX\), 
  the set of all tuples \(x\) in \(\LL\) such that 
  \(\sum{}x_ia_i\in{}V_\oY\) is \(a\)-definable, hence has a point in 
  \(\LL(a)\).)
\end{proof}

\point[Proof of~\ref{tan:main}\eqref{tan:mainexr}]\label{alg:mainexr}
Let \(\oG\) be the definable group in \(T_\Cc\) corresponding to it as an as 
an internal cover of \(\LL\) (Theorem~\ref{mod:internality}).  By 
Proposition~\ref{alg:fibismod}, \(\w\)  determines a definably closed subset 
\(\rA\) of an arbitrarily large model \(\rM\), such that 
\(\dcl{\rA\cup\LL(\rM)}=\rM\) and \(\LL(\rA)=\kk\).  Given a field extension 
\(\dcs{K}\), we may assume that \(\LL(\rM)\) contains \(\dcs{K}\).

For each object \(\oX\), 
\(V_\oX(\rA\cup\dcs{K})=\w(\oX)\Ten_\kk\dcs{K}\) (again by 
Proposition~\ref{alg:fibismod}), It follows by Proposition~\ref{mod:stronger} 
that for such a field \(\dcs{K}\), 
\(\oG(\dcl{\rA\cup\dcs{K}})=Aut(\dcl{\rA\cup\dcs{K}}/\dcs{K})\), and since by 
the definition of the theory, such automorphisms are the same as 
automorphisms of \(\w\Ten_\kk\dcs{K}\) as a tensor functor, we get that 
\(\oG(\dcl{\rA\cup\dcs{K}})=\Aut^\Ten(\w)(\dcs{K})\).

On the other hand, \(\rA\) satisfies Assumption~\ref{mod:nogpd}, hence we 
get a definable group \(\oG_\rA\) in \(\LL\) (i.e., in \(ACF\)), such 
that \(\oG_\rA(\dcs{K})=\oG(\dcl{\rA\cup\dcs{K}})\).  
By~\Cite[4.5]{acfgroups}, any such group is algebraic.\qed

\begin{remark}
  By inspecting the proof of Proposition~\ref{mod:stronger} and the 
  construction of \(\oG\) in this particular case, it is easy to extend the 
  above proof to the case when \(\dcs{K}\) is an integral domain. However, I 
  don't know how to use the same argument for more general \(\dcs{K}\).
\end{remark}

\point[Proof of~\ref{tan:main}\eqref{tan:mainrec}]\label{alg:mainrec}
Let \(\Cc=\Rep_\sch{H}\), where \(\sch{H}\) is an algebraic group over 
\(\kk\), which we view as a definable group \(\oH\) in \(ACF_\kk\).  Let 
\(\oG\) be the definable internality group in \(T_\w\) and \(\rA\) the 
definable subset corresponding to the (forgetful) fibre functor \(\w\), all 
as in~\ref{alg:mainexr}. We obtain a definable group \(\oG_\rA\), and 
the action of \(\sch{H}\) on its representations determines a map of group 
functors from \(\oH\) to \(\oG_\rA\), which by Beth definability is 
definable and therefore algebraic. This homomorphism is injective (on points 
in any field extension) since \(\oH\) has a faithful representation. We 
identify \(\oH\) with its image in \(\oG_\rA\).

By the Galois correspondence (Theorem~\ref{mod:internality}), to prove that 
\(\oH=\oG_\rA\) it is enough to show that any \(\rA\)-definable element in 
\(T_\Cc^{eq}\) fixed by \(\oH(\rM)\) is also fixed by \(\oG_\rA(\rM)\), i.e., 
is \(0\)-definable.  By Proposition~\ref{alg:eqei}, 
\(T_\Cc^{eq}=\ti{T}_\Cc\).  If \(v\in{}V_\oX(\rA)\) is fixed by \(\oH\), then 
the map \(\LL\ra{}V_\oX\) given by \(1\mt{}v\) is a map of \(\oH\) 
representations, hence is given by a function symbol \(f\), so \(v=f(1)\) is 
\(0\)-definable.  Likewise, if \(p\in{}P_\oX(\rA)\) is fixed by \(\oH\), then 
it corresponds to a sub-representation \(l_p\subset{}V_\oX\) (over \(\kk\)), 
hence is given by a predicate, so again \(p\) is \(0\)-definable.

This proves that the map from \(\oH\) to \(\oG_\rA\) is bijective on 
points in any field. Any such algebraic map (in characteristic \(0\)) is an 
isomorphism (\Cite[11.4]{Waterhouse}).\qed

\point[Proof of~\ref{tan:main}\eqref{tan:equiv} (sketch)]\label{alg:equiv}
Let \(T_\Cc\), \(\oG\), \(\rA\) and \(\oG_\rA\) be as above, and 
let \(\Vec_\LL\) be the category of definable \(\LL\)-vector spaces 
interpretable in \(T_\Cc\), with definable linear maps between them as 
morphisms. This is clearly a tensor category, and \(\oX\mt{}V_\oX\) 
is an exact faithful \(\kk\)-linear tensor functor from \(\Cc\) to 
\(\Vec_\LL\).

On the other hand, a representation of \(\oG_\rA\) is given by some definable 
action in \(ACF\), and by Proposition~\ref{mod:action}, this action comes 
from some action of \(\oG\) on a definable set in \(T_\Cc\), which must 
therefore be a vector space over \(\LL\). Conversely, each linear action of 
\(\oG\) on an object in \(\Vec_\LL\) gives a representation of \(\oG_\rA\) 
after taking \(\rA\)-points, and similarly for morphisms.  Thus, we get an 
equivalence (of \(\kk\)-linear tensor categories) between \(\Vec_\LL\) and 
\(\Rep_{\oG_\rA}\).

It remains to show that the functor \(\oX\mt{}V_\oX\) from 
\(\Cc\) to \(\Vec_\LL\) is full and surjective on isomorphism classes, 
i.e., that any definable \(\LL\)-vector space is essentially of the form 
\(V_\oX\), and that any map between is of the form \(v_f\). To prove the 
first, we consider the group \(\oG\) and its torsor \(\oX\). These 
are defined by some \(0\)-definable subspaces of powers of \(V_{\Co{Q}}\) 
(where \(Q\) is a tensor generator.) By examining the explicit definition of 
\(\oG\) and \(\oX\), one concludes that these \(0\)-definable 
subspaces in fact come from \(\Cc\). We omit the proof, since it is very 
similar to the original algebraic one. Once this is know, given any other 
representation \(\oD\), the set \(\oX_\oD\) in Proposition~\ref{mod:action} 
is explicitly given by tensor operations, and so comes from \(\Cc\) as well.  
Then proof for morphisms is again similar.\qed

\begin{cor}
  Quantifier elimination in \(T_\Cc\) holds without parameters: any definable 
  set is quantifier-free without parameters.
\end{cor}
\begin{proof}
  This holds anyway with parameters from the definable closure of the empty 
  set. We now know that any such parameter is given by a term in the 
  language.
\end{proof}

\begin{remark}\label{alg:deligne1}
  In~\Cite{groupoids}, an equivalence is constructed between internal covers 
  of a complete theory \(\cat{T}\), and connected definable groupoids in 
  \(\cat{T}\). It follows from Theorem~2.8 there that if \(T_\Cc\) is 
  consistent and the induced structure on \(\LL\) is precisely \(ACF_\kk\), 
  then the associated groupoid is connected. It also follows that \(T_\Cc\) 
  itself is complete.  It then follows from lemma~2.4 in the same paper that 
  any two models of \(T_\Cc\) are isomorphic. Given two fibre functors of 
  \(\Cc\) (over some extension field), extending both to a model of \(T_\Cc\) 
  thus shows that the fibre functors are locally isomorphic.  Conversely, if 
  \(T_\Cc\) induces new structure on \(\LL\), it is easy to construct fibre 
  functors which are not locally isomorphic (cf.~\Cite[3.15]{LNM900II}).

  The main result of~\Cite{Deligne} can be viewed as stating that the induced 
  structure on \(\LL\) is indeed \(ACF_\kk\). More precisely, given a fibre 
  functor \(\w\) over an (affine) scheme \(\sch{S}\) over \(\kk\), Deligne 
  constructs a groupoid scheme \(\Aut_\kk^\Ten(\w)\), with object scheme 
  \(\sch{S}\).  The main result of~\Cite{Deligne} (Theorem~1.12) states that 
  this groupoid scheme is faithfully flat over \(\sch{S}\x\sch{S}\) (hence 
  connected). A \(\dcs{K}\)-valued point of \(\sch{S}\) (for some field 
  \(\dcs{K}\)) determines a fibre functor \(\w_\dcs{K}\) over \(\dcs{K}\).  
  Viewing this fibre functor as subset of a model, a choice of a basis for a 
  tensor generator determines a type over \(\LL\), and therefore an object of 
  the groupoid corresponding to \(T_\Cc\) (choosing a different basis amounts 
  to choosing a different object which is isomorphic over \(\dcs{K}\) in the 
  groupoid). It is clear from definition of \(\Aut_\kk^\Ten(\w)\) that this 
  process gives an equivalence from this groupoid to the groupoid associated 
  with the internal cover \(T_\Cc\). Thus, Deligne's theorem implies that 
  this groupoid is connected, and therefore that the induced structure on 
  \(\LL\) is that of \(ACF_\kk\). It seems plausible that there should be a 
  direct model theoretic proof of this result, but I could not find it. We 
  note that the non-existence of new structure also directly implies (by the 
  existence of prime models) \Cite[corollary~6.20]{Deligne}, which states 
  that \(\Cc\) has a fibre functor over the algebraic closure of \(\kk\). 

  We note also that any definable groupoid in \(ACF_\kk\) is equivalent to a 
  groupoid scheme (any such groupoid is equivalent to one with a finite set 
  of objects, and the automorphism group of each object is algebraic). The 
  category \(\Cc\) of representations of the (gerb associated with the) 
  groupoid is a tensor category satisfying \(End(\1)=\kk\) 
  (\Cite[section~3]{LNM900II}), and it is easy to see (using the same methods 
  as in~\ref{alg:mainrec}) that the groupoid associated with the internal 
  cover \(T_\Cc\) is equivalent to the original one. Therefore, we obtain 
  that any internal cover of \(ACF_\kk\) has (up to equivalence) the form 
  \(T_\Cc\) for some Tannakian category over \(\kk\).
\end{remark}

\begin{remark}\label{alg:deligne2}
  We interpret model theoretically two additional result of~\Cite{Deligne}.  
  In section~7, it is shown that in our context (i.e., \(char(\kk)=0\) and 
  \(\rk(\oX)\) is a natural number), \(\Cc\) always has a fibre functor. In 
  light of Proposition~\ref{alg:fibismod}, this result asserts that \(T_\Cc\) 
  is consistent for any such \(\Cc\).

  In section~8, Deligne defines, for a Tannakian category \(\Cc\) (and even 
  somewhat more generally), a \(\Cc\)-group \(\pi(\Cc)\), called the 
  \emph{fundamental group} of \(\Cc\) (a \(\Cc\)-group is a commutative Hopf 
  algebra object in \(\Ind{\Cc}\); similarly for affine \(\Cc\)-schemes).  
  The basic idea is that the concepts associated with  neutral Tannakian 
  categories, also make sense for ``fibre functors'' into other tensor 
  categories, and \(\pi(\Cc)\) is obtain by reconstructing a group from the 
  identity functor.

  Thus, \(\pi(\Cc)\) comes with an action on each object, commuting with all 
  morphisms.  It thus represents the tensor automorphisms of the identity 
  functor from \(\Cc\) to itself, in the sense that for any affine 
  \(\Cc\)-scheme \(\oX\), the action identifies \(\pi(\Cc)(\oX)\) with the 
  group of tensor automorphisms of the functor \(\oA\mt\oA\Ten\oX\) (from 
  \(\Cc\) to the tensor category of vector bundles over \(\oX\)). It has the 
  property that for any ``real'' fibre functor \(\w\) over a scheme 
  \(\sch{S}\), applying \(\w\) to the action identifies \(\w(\pi(\Cc))\) with 
  the group scheme \(\Aut_\sch{S}^\Ten(\w)\) constructed.

  Using either the explicit construction, or the properties above, it is 
  clear that in terms of the theory \(T_\Cc\), \(\pi(\Cc)\) is nothing but 
  the definable automorphism group (more precisely, the Hopf algebra object 
  that defines it maps to an ind-definable Hopf algebra in \(T_\Cc\), which 
  is definably isomorphic to the ind-definable Hopf algebra of functions on 
  the internality group).
\end{remark}

\begin{comment}
This also compares with~\ref{mon:main}: the group \(G\) is obtained by 
considering the automorphisms of (essentially) the identity functor, rather 
than of the ``fibre functor'' \(\ftr{F}\) there.
\end{comment}

\begin{remark}\label{alg:galois}
  One can recover in a similar manner (and somewhat more easily) 
  Grothendieck's approach to usual Galois theory 
  (cf.~\Cite[Expos\'e~V.4]{SGA1}). Briefly, given a category \(\Cc\) with a 
  ``fibre functor'' \(\ftr{F}\) into the category of finite sets, satisfying 
  conditions (G1)--(G6), one constructs as above a theory \(T_\Cc\) with a 
  sort \(V_\oX\) for each object, and a function symbol for each morphism.  
  \(T_\Cc\) is then the theory of \(\ftr{F}\) viewed as a structure in this 
  language. Since every sort is finite, they are all internal to \(\2\) (the 
  co-product of the terminal object \(\1\) with itself). Conditions 
  (G1)--(G5) ensure that any definable set in fact comes from \(\Cc\), and 
  (G2), (G5) ensure that \(T_\Cc\) has elimination of imaginaries. The Galois 
  objects \(P_i\) that appear in the proof are precisely the \(1\)-types of 
  internality parameters that appear in the model theoretic construction of 
  Galois group.
\end{remark}

\section{Differential Tannakian categories}\label{dif}
Our purpose in this section is to define differential tensor categories, and 
to give a model theoretic proof of the basic theorem, corresponding such 
categories, endowed with a suitably defined fibre functor, with linear 
differential algebraic groups. The method is completely analogous to that in 
the previous section.

Throughout this section, \(\kk\) is a field of characteristic \(0\).

\point[Prolongations of abelian categories]
We assume that in a tensor category \((\Cc,\Ten)\), the functor \(\Ten\) 
is exact; this is automatic if \(\Cc\) is rigid 
(see~\Cite[prop.~1.16]{LNM900II}.)

\begin{defns}
Let \(\Cc\) be a \(\kk\)-linear category. The \Def{prolongation} 
\(\DD(\Cc)\) of \(\Cc\) is defined as follows: The objects are exact 
sequences \(\oX:=\Ob{\oX}\) of \(\Cc\), and the morphisms between 
such objects are morphisms of exact sequences whose two \(\oX_0\) parts 
coincide.
\end{defns}

An exact functor \(F:\Cc_1\ra\Cc_2\) gives rise to an induced functor 
\(\DD(F):\DD(\Cc_1)\ra\DD(\Cc_2)\).  We denote by \(\Pi_i\) 
(\(i=0,1\)) the functors from \(\DD(\Cc)\) to \(\Cc\) assigning 
\(\oX_i\) to the object \(\Ob{\oX}\) of \(\DD(\Cc)\) (thus there 
is an exact sequence \(\Ob{\Pi}\).) \(\Pi_i(\oX)\) is also abbreviated as 
\(\oX_i\), and \(\oX\) is said to be over \(\oX_0\) (and 
similarly for morphisms.)

\begin{remarks}
We note that \(\DD(\Cc)\) can be viewed as the full subcategory of the 
category of ``differential objects'' in \(\Cc\), consisting of objects 
whose homology is \(0\). A differential object is a pair \((\oX,\phi)\) 
where \(\oX\) is an object of \(\Cc\) and \(\phi\) is an endomorphism 
of \(\oX\) with \(\phi^2=0\). A morphism is a morphism in \(\Cc\) 
that commutes with \(\phi\), and the homology is \(\ker(\phi)/\im(\phi)\). 
This is the same as the category of \(\kk[\eps]\)-modules in \(\Cc\), in 
the sense of~\Cite[p.~155]{LNM900II} (where \(\eps^2=0\)). The advantage of 
this category is that it is again \(\kk\)-linear. However, I don't know how 
to extend the tensor structure (defined below) to this whole category (in 
particular, the tensor structure defined there does not seem to coincide with 
ours).
\end{remarks}

\subpoint
Let \(\oA\) and \(\oB\) be two objects over \(\oX_0\). Their 
\Def{Yoneda sum} \(\oA\YS\oB\) is a new object over \(\oX_0\), 
defined as follows (this is the addition in Yoneda's description of 
\(Ext^1(\oX_0,\oX_0)\)): the combined map 
\(\oX_0\x\oX_0\ra\oA_1\x\oB_1\) factors through 
\(\oA_1\x_{\oX_0}\oB_1\), and together with the map 
\(\oX_0\ra[1,-1]\oX_0\x\oX_0\) gives rise to a 
map\(f:\oX_0\ra\oA_1\x_{\oX_0}\oB_1\).  Let \(\oW_1\) be the co-kernel of 
this map. The map \(f\) composed with the projection from 
\(\oA_1\x_{\oX_0}\oB_1\) to \(\oX_0\) is \(0\), so we obtain an induced map 
\(p:\oW_1\ra\oX_1\).  The diagonal inclusion \(\Delta\) of \(\oX_0\) in 
\(\oW_1\) together with \(p\) give rise to an exact sequence 
\(0\ra\oX_0\ra[\Delta]\oW_1\ra[p]\oX_0\ra{}0\), which is the required 
object.

For any object \(\oA\) of \(\DD(\Cc)\), we denote by \(T(\oA)\) 
the object obtained by negating all arrows that appear in \(\oA\).

\subpoint[Tensor structure]
Let \((\Cc,\Ten,\phi_0,\psi_0)\) be a tensor category. An object 
\(\oX_0\) of \(\Cc\) gives rise to a functor from \(\DD(\Cc)\) to 
itself, by tensoring the exact sequence pointwise. Since we assumed \(\Ten\) 
to be exact, this functor also commutes with Yoneda sums: 
\((\oA\YS\oB)\Ten\oX_0\) is canonically isomorphic with 
\((\oA\Ten\oX_0)\YS(\oB\Ten\oX_0)\). Also, 
\(T(\oA)\Ten\oX_0\) is isomorphic to \(T(\oA\Ten\oX_0)\).

We endow \(\DD(\Cc)\) with a monoidal structure. The tensor product 
\(\oA\Ten\oB\) of the two \(\DD(\Cc)\) objects \(\oA\) and 
\(\oB\) is defined as follows: After tensoring the first with 
\(\oB_0\) and the second with \(\oA_0\), we obtain two objects over 
\(\oA_0\Ten\oB_0\). We now take their Yoneda sum.

We shall make use of the following exact sequence.

\begin{lemmas}\label{dif:exact}
  For any two objects \(\oA\) and \(\oB\) of \(\DD(\Cc)\), there is an exact 
  sequence
  \begin{equation}\label{eq:tenexact}
    0\ra{(\oA\Ten T(\oB))}_1\ra[i]
    \oA_1\Ten \oB_1\ra[\pi]{(\oA\Ten\oB)}_1\ra 0
  \end{equation}
  where \(\pi\) is the quotient of the map obtained from the maps 
  \(\pi_\oA\Ten{}1\) and \(1\Ten\pi_\oB\), and \(i\) is the restriction of 
  the map obtained from the maps \(i_\oA\Ten{}1\) and \(-1\Ten{}i_\oB\).
\end{lemmas}
\begin{proof}
  Exactness in the middle follows directly from the definitions. We prove 
  that \(\pi\) is surjective, the injectivity of \(i\) being similar. We 
  shall use the Mitchell embedding theorem (cf.~\Cite{freyd}), which reduces 
  the question to the case of abelian groups.

  We in fact prove that already the map
  \begin{equation*}
    \oA_1\Ten\oB_1\ra[\pi]\oA_0\Ten \oB_1\x_{\oA_0\Ten \oB_0}\oA_1\Ten 
    \oB_0=:U
  \end{equation*}
  is surjective. Let \(y\) be an element of \(U\), and let \(y_1\) and 
  \(y_2\) be its two projections to the components of \(U\). Since the map 
  \(\oA_1\Ten\oB_1\ra[\pi_\oA\Ten{}1]\oA_0\Ten\oB_1\) is 
  surjective, \(y_1\) can be lifted to an element \(\ti{y_1}\) of 
  \(\oA_1\Ten\oB_1\). We have that
  \begin{equation*}
    (\pi_\oA\Ten 1)((1\Ten\pi_\oB)(\ti{y_1}))=
    (1\Ten\pi_\oB)((\pi_\oA\Ten 1)(\ti{y_1}))=(1\Ten\pi_\oB)(y_1)=
    (\pi_\oA\Ten 1)(y_2)
  \end{equation*}
  Let \(z=(1\Ten\pi_\oB)(\ti{y_1})-y_2\). Since \(z\) is killed by 
  \(\pi_\oA\Ten{}1\), it comes from an element, also \(z\), of 
  \(\oA_0\Ten\oB_0\). Let \(\ti{z}\) be a lifting of \(z\) to 
  \(\oA_0\Ten\oB_1\), and denote by \(\ti{z}\) also its image in 
  \(\oA_1\Ten\oB_1\) under the inclusion \(i_\oA\Ten{}1\). Then 
  \(\ti{y_1}-\ti{z}\) is a lifting of \(y\).
\end{proof}

\subpoint
Let \(\oA,\oB,\oC\) be three objects of \(\DD(\Cc)\). The 
associativity constraint \(\phi_0\) of \(\Cc\) gives rise to an 
isomorphism of \((\oA\Ten\oB)\Ten\oC\) with the quotient of
\begin{equation*}
  \oA_1\Ten\oB_0\Ten\oC_0\x_{\oA_0\Ten\oB_0\Ten\oC_0}
  \oA_0\Ten\oB_1\Ten\oC_0\x_{\oA_0\Ten\oB_0\Ten\oC_0}
  \oA_0\Ten\oB_0\Ten\oC_1
\end{equation*}
that identifies the three natural inclusions of 
\(\oA_0\Ten\oB_0\Ten\oC_0\), and similarly for 
\(\oA\Ten(\oB\Ten\oC)\). We thus get an associativity constraint 
\(\phi\) on \(\DD(\Cc)\), over \(\phi_0\).

Likewise, the commutativity constraint \(\psi_0\) induces a commutativity 
constraint \(\psi\) on \(\DD(\Cc)\) over \(\psi_0\).

\begin{props}
  The data \((\DD(\Cc),\Ten,\phi,\psi)\) as defined above forms a symmetric 
  monoidal category, and \(\Pi_0\) is a monoidal functor. It is rigid if 
  \(\Cc\) is rigid.
\end{props}
\begin{proof}
  We define the additional data. Verification of the axioms reduces, as in 
  Lemma~\ref{dif:exact}, to the case of abelian groups, where it is easy.

  Let \(u:\1_0\ra\1_0\Ten\1_0\) be an identity object of \(\Cc\). We set  
  \(\1=0\ra\1_0\ra\1_0\oplus\1_0\ra\1_0\ra{}0\). For any object \(\oA\) 
  of \(\DD(\Cc)\), \(\1\Ten\oA\) is identified via \(u\) with 
  \begin{equation*}
    0\ra\oA_0\ra 
    (\oA_1\x_{\oA_0}(\oA_0\oplus\oA_0))/\oA_0\ra 
    \oA_0\ra 0
  \end{equation*}
  This object is canonically isomorphic (over \(\Cc\)) to \(\oA\), 
  and so \(\1\) acquires a structure of an identity object.

  Assume that \(\Cc\) is rigid. For an object \(\oA\) of 
  \(\DD(\Cc)\), we set \(\Co{\oA}\) to be the dual exact sequence 
  \(\Ob{\Co{\oA}}\). We define an evaluation map 
  \(\oA\Ten\Co{\oA}\ra\1\) as follows: We need to define two maps 
  from \(\oA_0\Ten\Co{\oA_1}\x\oA_1\Ten\Co{\oA_0}\) to 
  \(\1_0\), that agree on the two inclusions of 
  \(\oA_0\Ten\Co{\oA_0}\), and such that the resulting map restricts 
  to the evaluation on \(\oA_0\Ten\Co{\oA_0}\).

  To construct the first map, we consider the exact 
  sequence~\eqref{eq:tenexact}, for \(\oB=\Co{\oA}\). We claim that 
  the evaluation map on \(\oA_1\Ten\Co{\oA_1}\) restricts to \(0\) 
  when composed with \(i\). To prove this, it is enough to show that the pair 
  of maps obtained from \(ev_{\oA_1}\) by composition with 
  \(i_\oA\Ten{}1\) and \(-1\Ten{}i_{\Co{\oA}}\) comes from a map 
  \(\oA_0\Ten\Co{\oA_0}\ra\1_0\).  However, under the adjunction, 
  this pair of maps corresponds to \((i_\oA,-\pi_{\oA})\), and so 
  comes from the identity map on \(\oA_0\). It follows that 
  \(ev_{\oA_1}\) induces a map on \({(\oA\Ten\Co{\oA})}_1\), 
  which is the required map. The second map is obtained by projecting to 
  \(\oA_0\Ten\Co{\oA_0}\), and using the evaluation map on 
  \(\oA_0\). By definition, this second map commutes with the projections 
  to \(\oA_0\Ten\Co{\oA_0}\) and the second coordinate of \(\1\), 
  restricting to the evaluation on \(\oA_0\). To prove that the first map 
  restricts to the evaluation as well, we note that there is a commutative 
  diagram
  \begin{equation*}
    \xymatrix{
    {(\oA\Ten\Co{\oA})}_1\ar[r]^i\ar[d]^{\pi_{\oA\Ten\Co{\oA}}}& 
    \oA_1\Ten\Co{\oA_1}\ar[d]^\pi\\
    \oA_0\Ten\Co{\oA_0}\ar[r]_{i_{\oA\Ten\Co{\oA}}} & {(\oA\Ten\Co{\oA})}_1
    }
  \end{equation*}
  where \(i\) is the (restriction of the) map obtained from the two maps 
  \(i_\oA\Ten{}1\) and \(1\Ten{}i_{\Co{\oA}}\).

  Since \(\pi_{\oA\Ten\Co{\oA}}\) is surjective, it is therefore 
  enough to prove that the maps \(ev_{\oA_1}\circ{}i\) and 
  \(ev_{\oA_0}\circ\pi_\oA\) coincide.  This is indeed the case, 
  since they both correspond to the inclusion of \(\oA_0\) in 
  \(\oA_1\).
\end{proof}

\point[Differential tensor categories]
\begin{defns}
  A \Def{differential structure} on a tensor category \(\Cc\) is a tensor 
  functor \(\ftr{D}\) from \(\Cc\) to \(\DD(\Cc)\) which is a section of 
  \(\Pi_0\). If \(\ftr{D}_1\) and \(\ftr{D}_2\) are two differential 
  structures on \(\Cc\), a morphism from \(\ftr{D}_1\) to \(\ftr{D}_2\) is a 
  morphism of tensor functors that induces the identity morphism under 
  \(\Pi_0\). A \Def{differential tensor category} is a tensor category 
  together with a differential structure.
\end{defns}

Let \(\ftr{D}\) be a differential structure on \(\Cc\). Since \(\ftr{D}\) 
is a section of \(\Pi_0\), it is determined by \(\d=\Pi_1\circ{}D\). In other 
words, on the abelian level, it is given by a functor 
\(\d:\Cc\ra\Cc\), together with an exact sequence 
\(0\ra\Id\ra\d\ra\Id\ra{}0\). However, this description does not include the 
tensor structure. We also note that \(\d\) is necessarily exact.

\subpoint\label{dif:dif2der}
Let \((\Cc,\ftr{D})\) be a differential tensor category, let 
\(\d=\Pi_1\circ\ftr{D}\), and let \(A=End(\1)\). Recall that for any object 
\(\oX\), \(End(\oX)\) is an \(A\)-algebra.  The functor \(\d\) 
defines another ring homomorphism \(\d_{\1}:A\ra{}End(\d(\1))\).  Given 
\(a\in{}A\), the morphism \(\d_{\1}(a)-a\) in \(End(\d(\1))\) restricts to 
\(0\) on \(\1\), and thus induces a morphism from \(\1\) to \(\d(\1)\).  
Similarly, its composition with the projection \(\d(\1)\ra\1\) is \(0\), so 
it factors through \(\1\). We thus get a new element \(a'\) of \(A\).
\begin{claims}
  The map \(a\mt{}a'\) of~\ref{dif:dif2der} is a derivation on \(A\).
\end{claims}
\begin{proof}
  We need to show that given elements \(a,b\in{}A\), the maps \(\d(ab)-ab\) 
  and \((\d(a)-a)b+a(\d(b)-b)\) coincide on \(\1\). This follows from the 
  formula \(\d(ab)-ab=\d(a)(\d(b)-b)+(\d(a)-a)b\), together with the fact 
  that \(\d(a)(\d(b)-b)\) induces \(a(\d(b)-b)\) on \(\1\).
\end{proof}

\begin{examples}\label{dif:der2dif}
  Let \(\Cc\) be the tensor category \(\Vec_\kk\) of finite dimensional 
  vector spaces over a field \(\kk\). Given a derivative \('\) on \(\kk\), we 
  construct a differential structure on \(\Cc\) as follows: For a vector 
  space \(\oX\), define \(d(\oX)=\dd\ten\oX\), where \(\dd\) is the vector 
  space with basis \(1,\d\), and \(\ten\) is the tensor product with respect 
  to the \emph{right} vector space structure on \(\dd\), given by 
  \(1\cdot{}a=a\cdot{}1\) and \(\d\cdot{}a=a'\cdot{}1+a\cdot\d\). The exact 
  sequence \(\ftr{D}(\oX)\) is defined by \(x\mt{}1\ten{}x\), 
  \(1\ten{}x\mt{}0\) and \(\d\ten{}x\mt{}x\), for any \(x\in\oX\).  If 
  \(T:\oX\ra\oY\) is a linear map, \(d(T)=1\ten{}T\). We shall write \(x\) 
  for \(1\ten{}x\) and \(\d{}x\) for \(\d\ten{}x\). The structure of a tensor 
  functor is obtained by sending \(\d(x\Ten{}y)\in{}d(\oX\Ten\oY)\) to the 
  image of \(\d(x)\Ten{}y\oplus{}x\Ten\d(y)\) in 
  \((\ftr{D}(\oX)\Ten\ftr{D}(\oY))_1\).
\end{examples}

\begin{claims}
  The constructions in~\ref{dif:der2dif} and in~\ref{dif:dif2der} give a 
  bijective correspondence between derivatives on \(\kk\) and isomorphism 
  classes of differential structures on \(\Vec_\kk\).
\end{claims}
\begin{proof}
  If \(\ftr{D}_1\) and \(\ftr{D}_2\) are two differential structures, then 
  \(\ftr{D}_1(\1_0)\) and \(\ftr{D}_2(\1_0)\) are both identity objects, and 
  are therefore canonically isomorphic to the same object \(\1\). If 
  \(\ftr{D}_1\) and \(\ftr{D}_2\) are isomorphic, then the maps 
  \(d_i:End(\1_0)\ra{}End(\1)\) are conjugate, and therefore equal, since 
  \(End(\1)\) is commutative.

  It is clear from the definition that the derivative on \(\kk\) obtained 
  from the differential structure associated with a derivative is the 
  original one. Conversely, if \(\ftr{D}_1\) and \(\ftr{D}_2\) are two 
  differential structures that give the same derivative on \(\kk\), then we 
  may identify \(\ftr{D}_1(\1_0)\) and \(\ftr{D}_2(\1_0)\). Under this 
  identification, we get that the maps \(d_i\) are the same. But the functors 
  \(\ftr{D}_i\) are determined by \(d_i\).
\end{proof}

\subpoint
We now come to the definition of functors between differential tensor 
categories. For simplicity, we shall only define (and use) \emph{exact} such 
functors.  

Let \(\w:\Cc\ra\cat{D}\) be an exact functor between abelian categories.  
There is an induced functor \(\DD(\w):\DD(\Cc)\ra\DD(\cat{D})\), given by 
applying \(\w\) to each term. If \(\Cc\) and \(\cat{D}\) are tensor 
categories, the structure of a tensor functor on \(\w\) gives rise to a 
similar structure on \(\DD(\w)\) (again, since \(\w\) is exact.)  If 
\(t:\w_1\ra\w_2\) is a (tensor) morphism of functors, we likewise get an 
induced morphism \(\DD(t):\DD(\w_1)\ra\DD(\w_2)\).

\begin{defns}
Let \((\Cc_1,\ftr{D}_1)\) and \((\Cc_2,\ftr{D}_2)\) be two differential 
tensor categories.
\begin{enumerate}
  \item
    A \Def{differential tensor functor} from \(\Cc_1\) to \(\Cc_2\) is an 
    exact tensor functor \(\w\) from \(\Cc_1\) to \(\Cc_2\), together with an 
    isomorphism of tensor functors 
    \(r:\DD(\w)\circ\ftr{D}_1\ra\ftr{D}_2\circ\w\), such that the 
    ``restriction'' \(\Pi_0\odot{}r:\w\ra\w\), obtained by composing with 
    \(\Pi_0\) on \(\DD(\Cc_2)\), is the identity.

  \item
    A morphism between two such differential tensor functors \((\w_1,r_1)\) 
    and \((\w_2,r_2)\) is a morphism \(t\) between them as tensor functors 
    such that the following diagram (of tensor functors and tensor maps 
    between them) commutes:
    \begin{equation}\label{eq:diffmorphism}
      \xymatrix{
      \DD(\w_1)\circ \ftr{D}_1\ar[r]^{r_1}\ar[d]^{\DD(t)\odot \ftr{D}_1} &
      \ftr{D}_2\circ\w_1 \ar[d]^{\ftr{D}_2\odot t} \\
      \DD(\w_2)\circ \ftr{D}_1\ar[r]^{r_2} & \ftr{D}_2\circ\w_2
      }
    \end{equation}
    where \(\ftr{D}_2\odot{}t\) is the map from \(\ftr{D}_2\circ\w_1\) to 
    \(\ftr{D}_2\circ\w_2\) obtained by applying \(\ftr{D}_2\) to \(t\) 
    ``pointwise''.
\end{enumerate}
\end{defns}

\subpoint\label{dif:autw}
Given a differential tensor functor \(\w\), we denote by \(Aut^{\d}(\w)\) the 
group of automorphisms of \(\w\).

If \(\Cc\) is a differential tensor category, and \(\kk=End(\1)\) is a 
field, a \(\kk\)-linear differential tensor functor \(\w\) into \(\Vec_\kk\) 
(with the induced differential structure) is called a \Def{differential fibre 
functor}. Analogously to the algebraic case, we will say that \(\Cc\) is a 
\Def{neutral differential Tannakian category} if such a fibre functor exists, 
and that \(\w\) neutralises \(\Cc\). Since we will not define more general 
differential Tannakian categories, the adjective ``neutral'' will be dropped.

As in the algebraic case, given a differential fibre functor \(\w\) on 
\(\Cc\), we denote by \(\Aut^\d(\w)\) the functor from differential 
\(\kk\)-algebras to groups assigning to a differential algebra \(A\) the 
group \(Aut^{\d}(A\Ten\w)\).

\subpoint\label{dif:universal}
Given a \(\kk\) vector space \(V\), the map \(d:V\ra\dd\ten{}V\) given by 
\(v\mt\d{}v\) is a derivation, in the sense that \(d(av)=a'v+ad(v)\) (where 
\(V\) is identified with its image in \(\dd\ten{}V\).) It is universal for 
this property: any pair \((i,d):V\ra{}W\), where \(i\) is linear, and \(d\) 
is a derivation with respect to \(i\) factors through it.

Therefore, a fibre functor on \((\Cc,\ftr{D})\) is a fibre functor \(\w\) 
in the sense of tensor categories, together with a functorial derivation 
\(d_{\oX_0}:\w(\oX_0)\ra\w(\d\oX_0)\) (where 
\(\d\oX_0=\ftr{D}(\oX_0)_1\)), satisfying the Leibniz rule with 
respect to the tensor product (and additional conditions). The condition that 
the restriction to \(\w\) is the identity corresponds to the derivation being 
relative to the canonical inclusion of \(\w(\oX_0)\) in 
\(\w(\d\oX_0)\) given by the differential structure.

Similarly, a differential automorphism of \(\w\) is an automorphism \(t\) of 
\(\w\) as a tensor functor, with the additional condition that for any object 
\(\oX_0\), the diagram
\begin{equation}\label{eq:fibauto}
  \xymatrix{
  \w(\oX_0)\ar[r]^{d_{\oX_0}}\ar[d]^{t_{\oX_0}} & 
  \w(\d\oX_0)\ar[d]^{t_{\d\oX_0}} \\
  \w(\oX_0)\ar[r]^{d_{\oX_0}}             & \w(\d\oX_0)
  }
\end{equation}
commutes. Thus the condition~\eqref{eq:diffmorphism} really is about 
preservation of the differentiation.

\subpoint[Derivations]\label{dif:derivation}
More generally, we define a \Def{derivation} on an object \(\oX_0\) of 
\(\Cc\) to be a morphism \(d:\oX_1:=\d(\oX_0)\ra\oX_0\) such that 
the composition \(\oX_0\ra\oX_1\ra[d]\oX_0\) is the identity.

Given two derivations \(d_X:\oX_1\ra\oX_0\) and 
\(d_Y:\oY_1\ra\oY_0\), we define  the combined derivation 
\(d=d_X\Ten{}d_Y:(\oX\Ten\oY)_1=(\oX_0\Ten\oY_0)_1\ra\oX_0\Ten\oY_0\)
to be the restriction of \(d_X\Ten\Id\oplus\Id\Ten{}d_Y\) to 
\((\oX\Ten\oY)_1\) (this makes sense, since both are the identity on 
\(\oX_0\Ten\oY_0\)). In \(\Vec_\kk\) this corresponds to a real 
derivation, in the sense that it gives a map 
\(d_0:\oX_0\Ten\oY_0\ra\oX_0\Ten\oY_0\) which is a 
derivation over \(\kk\) and also 
\(d_0(x\Ten{}y)=d_X(x)\Ten{}y+x\Ten{}d_Y(y)\).

There is a canonical derivation \(d_1\) on \(\1_0\) given by the first 
projection \(\1_1=\1_0\oplus\1_0\ra\1_0\). It has the property that 
\(d_1\Ten{}d=d\Ten{}d_1\) for any derivation \(d\) on any \(\oX_0\),  
under the canonical identification of \(\1\Ten\oX_0\) and 
\(\oX_0\Ten\1\) with \(\oX_0\).

\point[Differential algebraic groups]
In this sub-section, we recall and summarise some definitions and basic facts  
from differential algebraic geometry (developed by~\Cite{Kolchin}) and linear 
differential algebraic groups (studied by~\Cite{cassidy}). We show that the 
category of differential representations of such a group is a differential 
tensor category in our sense.

\subpoint
Let \(\dcs{K}\) be a differential field (i.e., a field endowed with a 
derivation).  We recall (\Cite{Kolchin}) that a \Def{Kolchin closed} subset 
(of a affine \(n\)-space) is given by a collection of polynomial (ordinary) 
differential equations in variables \(x_1,\dots,x_n\), i.e., polynomial 
equations in variables \(\delta^i(x_j)\), for \(i\ge{}0\). Such a collection 
determines a set of points (solutions) in any differential field extension of 
\(\dcs{K}\).  As with algebraic varieties, it is possible to study these sets 
by considering points in a fixed field, provided that it is 
\Def{differentially closed}. The Kolchin closed sets form a basis of closed 
sets for a Noetherian topology.  Morphisms are also given by differential 
polynomials. A differential algebraic group is a group object in this 
category.

More generally, it is possible to consider a differential 
\(\dcs{K}\)-algebra, i.e., a \(\dcs{K}\)-algebra with a vector field 
extending the derivation on \(\dcs{K}\), and develop these notions there.

\subpoint
By a \Def{linear} differential algebraic group, we mean a differential 
algebraic group which is represented by a differential Hopf algebra. A 
differential algebraic group which is affine as a differential algebraic 
variety need not be linear in this sense, since a morphism of affine 
differential varieties need not correspond to a map of differential algebras.  
Any linear differential algebraic group has a faithful representation. All 
these results appear in~\Cite{cassidy}, along with an example of an affine 
non-linear group. In~\Cite{cassidy2} it is shown that any representation of a 
\emph{linear} group (and more generally, any morphism of linear groups) does 
correspond to a map of differential algebras.

\subpoint[Differential representations]\label{dif:rep}
Let \(\sch{G}\) be a linear differential algebraic group over a differential 
field \(\kk\). A representation of \(\sch{G}\) is given by a finite 
dimensional vector space \(V\) over \(\kk\), together with a morphism 
\(\sch{G}\ra{}GL(V)\). A map of representations is a linear transformation 
that gives a map of group representations for each differential 
\(\kk\)-algebra.  The category of all such representations is denoted 
\(\Rep_\sch{G}\).

We endow \(\Rep_\sch{G}\) with a differential structure in the same way as 
for vector spaces. If \(V\) is a representation of \(\sch{G}\), assigning 
\(gv\) to \((g,v)\), then the action of \(\sch{G}\) on \(\dd\ten{}V\) is 
given by \((g,x\ten{}v)\mt{}x\ten{}gv\).  With this differential structure, 
the forgetful functor \(\w\) into \(\Vec_\kk\) has an obvious structure of a 
differential tensor functor.

A differential automorphism \(t\) of \(\w\) is given by a collection of 
vector space automorphisms \(t_V\), for any representation \(V\) of \(G\). 
The commutativity condition~\eqref{eq:fibauto} above translates to the 
condition that \(t_{\dd\ten{}V}=1\ten{}t_V\).

In particular, given a differential \(\kk\)-algebra \(A\), and 
\(g\in\sch{G}(A)\), action by \(g\) gives an automorphism of \(A\Ten\w\) as 
a differential tensor functor, since the action of \(g\) on \(\dd\ten{}V\) is 
deduced from its action on \(V\). Thus we get a map \(\sch{G}\ra\sch{G}_\w\).  
We shall prove in Theorem~\ref{dif:main} that the map is an isomorphism.

\begin{examples}
Let \(\sch{G}_m\) be the (differential) multiplicative group, and let 
\(\wbar{\sch{G}_m}\) be the multiplicative group of the constants (thus, as 
differential varieties, \(\sch{G}_m\) is given by the equation \(xy=1\), and 
\(\wbar{\sch{G}_m}\) is the subvariety given by \(x'=0\).) There is a 
differential algebraic group homomorphism \(\dlog\) from \(\sch{G}_m\) to 
\(\sch{G}_a\) (the additive group), sending \(x\) to \(x'/x\), and 
\(x\mt{}x'\) is a differential algebraic group endomorphism of \(\sch{G}_a\). 
Let \(V\) be the standard \(2\)-dimensional algebraic representation of 
\(\sch{G}_a\).  Using \(\dlog\) and the derivative, we get for any 
\(i\ge{}0\) a \(2\)-dimensional irreducible representation \(V_i\) of 
\(\sch{G}_m\), which are all unrelated in terms of the tensor structure (and 
unrelated with the non-trivial \(1\)-dimensional algebraic representations of 
\(\sch{G}_m\).)

However, if \(\oX\) is the \(\sch{G}_m\) representation corresponding to 
the identity map on \(\sch{G}_m\), an easy calculation shows that \(V_0\) is 
isomorphic to \(\d\oX\Ten\Co{\oX}\). Similarly, \(V_{i+1}\) is a 
quotient of \(\d{}V_i\).

The inclusion of \(\wbar{\sch{G}_m}\) in \(\sch{G}_m\) gives a functor from 
\(\Rep_{\sch{G}_m}\) to \(\Rep_{\wbar{\sch{G}_m}}\). But in 
\(\Rep_{\wbar{\sch{G}_m}}\), \(V_0\) is isomorphic to \(\1\oplus\1\) (and 
\(\d\oX\) to \(\oX\oplus\oX\).)
\end{examples}

\point[Differential schemes in \(\Cc\)]
We define affine differential schemes in a differential tensor category, and 
show that any object can be viewed as a ``differential affine space''. This 
is analogues to the notion of \(\Cc\)-schemes for tensor categories that 
appears in~\Cite{Deligne}. The main application is the proof of elimination 
of imaginaries in Proposition~\ref{dif:ei}.

\subpoint
We recall that the \(\Ten\) operation on \(\Cc\) extends canonically to 
\(\Ind{\Cc}\), making it again an abelian monoidal category. The 
prolongation \(\DD(\Ind{\Cc})\) can be identified with 
\(\Ind{\DD(\Cc)}\), and a differential structure on \(\Cc\) thus 
extends canonically to a differential structure on \(\Ind{\Cc}\).

Recall (\Cite[7.5]{Deligne}) that if \(\Cc\) is a tensor category, a 
\Def{ring} in \(\Ind{\Cc}\) is an object \(\oA\) of \(\Ind{\Cc}\) 
together with maps \(m:\oA\Ten\oA\ra\oA\) and \(u:\1\ra\oA\) 
satisfying the usual axioms.

\begin{defns}
  Let \((\Cc,\ftr{D})\) be a differential tensor category, and let \(\oA\) be 
  a commutative ring in \(\Ind{\Cc}\). A \Def{vector field} on \(\oA\) is a 
  derivation on \(\oA\) in the sense of~\ref{dif:derivation}, which commutes 
  with the product, and which restricts to the canonical derivation \(d_1\) 
  on \(\1\). A \Def{differential algebra} in \(\Ind{\Cc}\) is a commutative 
  ring in \(\Ind{\Cc}\) together with a vector field. An (affine) 
  \Def{differential scheme} in \(\Cc\) is a differential algebra in 
  \(\Ind{\Cc}\), viewed as an object in the opposite category.
\end{defns}

\subpoint[Higher derivations]\label{dif:higher}
Let \(\oX_0\) be an object of a differential tensor category 
\((\Cc,\ftr{D})\). As explained above, \(\ftr{D}(\oX_0)\) can be viewed 
as representing a universal derivation on \(\oX_0\). We now construct the 
analogue of higher derivatives. More precisely, we define, by induction for 
each \(n\ge{}0\), the following data:
\begin{enumerate}
  \item An object \(\oX_n\) (of \(\Cc\))
  \item A map \(q_n:\oX_{n-1}\ra\oX_n\)
  \item A map \(t_n:\d\oX_{n-1}\ra\oX_n\)
\end{enumerate}
such that \(q_{n+1}\circ{}t_n=t_{n+1}\circ\d(q_n)\). In the context of 
\(\Vec_\kk\), the data can be thought of as follows: \(\oX_n\) is the 
space of expressions \(v_0+\d{}v_1+\dots+\d^n{}v_n\) with 
\(v_j\in\oX_0\); \(q_n\) the inclusion of elements as above; the map 
\(t_n\) the linear map corresponding to the derivation.

For the base, setting \(\oX_{-1}=0\) determines all the data in an 
obvious way. Given \(\oX_n\), \(t_{n+1}:\d(\oX_n)\ra\oX_{n+1}\) 
is defined to be the co-equaliser of the following two maps:
\begin{equation}
  \xymatrix{
                  & \oX_n\ar[rd]^{i} & \\
      \d(\oX_{n-1})\ar[ru]^{t_n}\ar[rr]^{\d(q_n)} & & \d(\oX_n)
  }
\end{equation}
Where \(i\) is part of the structure of \(\ftr{D}(\oX_n)\). The map 
\(q_{n+1}\) is the composition \(t_{n+1}\circ{}i\). Clearly the commutativity 
condition is satisfied. We note that the two object we denote by 
\(\oX_1\) coincide, and the map \(q_1\) coincides with the map \(i\) for 
\(\ftr{D}(\oX_0)\). The map \(t_1\) is the identity.

\begin{defns}\label{dif:scheme}
  Let \(\oX_0\) be an object of of \(\Cc\). The \Def{differential scheme 
  associated with \(\oX_0\)}, denoted \(\Sch(\oX_0)\), is a differential 
  scheme in \(\Cc\) defined as follows: Let \(\oD\) be the ind-object defined 
  by the system \(\Co{\oX}_i\), with maps \(q_i\) (as in~\ref{dif:higher}).  
  The maps \(t_i\) there define a derivation \(t\) on \(\oD\). This 
  derivation induces a derivation on tensor powers of \(\oD\) (as 
  in~\ref{dif:derivation}), which descends to the symmetric powers. It is 
  easy to see that this determines a differential algebra structure on the 
  symmetric algebra on \(\oD\). We let \(\Sch(\oX_0)\) be the associated 
  scheme.
\end{defns}
A morphism in \(\Cc\) clearly determines a morphism of schemes on the 
associated schemes, making \(\Sch(-)\) into a functor.

\point[Model theory of differential fibre functors]
We now wish to prove statements analogous to the ones for algebraic Tannakian 
categories, using the same approach as in Section~\ref{alg}. We work with a 
fixed differential tensor category \((\Cc,\ftr{D})\), with \(\kk=End(\1)\) a 
field. We view  \(\kk\) as a differential field, with the differential 
structure induced from \(\ftr{D}\), as in~\ref{dif:dif2der}.

We will be using the theory \(DCF\) of differentially closed fields. We refer 
the reader to~\Cite{DCF} or~\Cite{DCFv} for more information.

\subpoint[The theory associated with a differential tensor 
category]\label{dif:theory}
The theory \(T_\Cc\) associated with the data above, as well as a 
differential field extension \(\dcs{K}\) of \(\kk\) is an expansion of the 
theory \(T_\Cc\) as defined in~\ref{alg:theory} by the following 
structure:
\begin{enumerate}
  \item
    \(\LL\) has an additional function symbol \('\), and the theory says that 
    \('\) is a derivation, and that \(\LL\) is a differentially closed field 
    (and with the restriction of \('\) to \(\dcs{K}\) as given).

  \item\label{dif:der}
    For every object \(\oX\), there is a function symbol 
    \(d_\oX:V_\oX\ra{}V_{\d(\oX)}\).  This function is a 
    derivation, in the sense that for any \(a\in\LL\) and 
    \(v\in{}V_\oX\),
    \begin{equation*}
      d_\oX(av)=a'V_{i_\oX}(v)+ad_\oX(v)
    \end{equation*}
    The theory furthermore says that \(d_\oX\) identifies 
    \(V_{\d(\oX)}\) with \(\dd\ten{}V_\oX\) (in any model), in the 
    sense of~\ref{dif:universal} (explicitly, it says that 
    \(V_{p_\oX}\circ{}d_\oX\) is the identity map.)

  \item\label{dif:comp}
    The maps \(d\) and \(b\) (from the tensor structure) are compatible with 
    the structure of tensor functor of \(\ftr{D}\): given objects \(\oX\) 
    and \(\oY\) of \(\Cc\), let 
    \(c_{\oX,\oY}:\d(\oX\Ten\oY)\ra{}(\d(\oX)\Ten\d(\oY))_1\) 
    be the isomorphism supplied with \(\ftr{D}\). Then we require that 
    \(V_{c_{\oX,\oY}}\circ{}d_{\oX\Ten\oY}\circ{}b_{\oX,\oY}\) 
    coincides with 
    \(b_{\d(\oX),\oY}\circ{}d_\oX\x{}1+b_{\oX,\d(\oY)}\circ{}1\x{}d_\oY\).
\end{enumerate}

\subpoint\label{dif:dcl}
Let \(\w\) be a differential fibre functor on \(\Cc\), and let 
\(L=\rM_\1\) be a differentially closed field containing \(\dcs{K}\). As 
in Section~\ref{alg}, we expand \(\rM_\1\) to a model \(\rM\) of \(T_\Cc\) by 
tensoring with \(L\).  The differential structure of \(\w\) gives (as 
in~\ref{dif:universal}) a universal derivation 
\(\w(\oX)\ra\w(\d(\oX))\), which extends uniquely to a (universal) 
derivation \((d_\oX)_\rM\) on \(\rM_\oX\).

As in the algebraic case, we get:
\begin{props}
  Assume that \(\Cc\) has a differential fibre functor.  Then \(T_\Cc\) is 
  consistent, and (in a model) \(\dcl{0}\cap\LL=\kk\). The theory \(T_\Cc\) 
  is stable, and \(\LL\) is a pure differentially closed field.
\end{props}
\begin{proof}
  Same as in Proposition~\ref{alg:stable} and Proposition~\ref{alg:fibismod}.
\end{proof}

\subpoint[Internality]
Since the differential \(T_\Cc\) is an expansion of the algebraic one 
with no new sorts, it is again an internal cover of \(\LL\).  Furthermore, if 
\(B\) is a basis for some \(V_\oX\), then \(B\cup{}d_\oX(B)\) is a 
basis for \(V_{\d(\oX)}\).  Therefore, if \(\Cc\) is generated as a 
differential tensor category by one object (in the sense that the objects 
\(\d^i\oX\) generate \(\Cc\) as a tensor category), then all of the 
sorts are internal using the same finite parameter. As usual, the general 
case is obtained by taking a limit of such.

\begin{theorems}\label{dif:main}
Let \(\Cc\) be a differential Tannakian category over \(\kk\), and let 
\(\w:\Cc\ra\Vec_\kk\) be a differential fibre functor (\ref{dif:autw}).
\begin{enumerate}
  \item\label{dif:mainexr}
    The functor \(\Aut^\d(\w)\) (restricted to differential fields) is 
    represented by a linear differential group \(\sch{G}\) over \(\kk\).

  \item\label{dif:equiv}
    The fibre functor \(\w\) factors through a differential tensor 
    equivalence \(\w:\Cc\ra\Rep_\sch{G}\).

  \item\label{dif:mainrec}
    If \(\Cc=\Rep_\sch{H}\) for some linear differential group \(\sch{H}\), 
    then the natural map \(\sch{H}\ra\sch{G}\) given in~\ref{dif:rep} is an 
    isomorphism.
\end{enumerate}
\end{theorems}

\begin{proof}
  The proof is completely analogous to the proof of the corresponding 
  statement~\ref{tan:main}, which is given, respectively, 
  in~\ref{alg:mainexr},~\ref{alg:equiv} and~\ref{alg:mainrec}. We only 
  mention the differences.

  For~(\ref{dif:mainexr}), the main point is again that 
  \(\Aut^\d(\w)(\dcs{K})\) is isomorphic to \(\oG_\w(\dcs{K})\) (functorially 
  in \(\dcs{K}\)), where \(\oG_\w\) is a definable copy of the internality 
  group inside \(\LL\). The sort \(\LL\) is now a pure differentially closed 
  field, so the result follows from the fact that any definable group in 
  \(DCF\) is differential algebraic (\Cite{pillayDAG}).

  For~(\ref{dif:equiv}), the argument in~\ref{alg:equiv} goes through 
  without a change.

  For~(\ref{dif:mainrec}), again the proof in~\ref{alg:mainrec} applies, 
  once we classify the imaginaries in \(T_\Cc\). This is the content of 
  Proposition~\ref{dif:ei}.
\end{proof}

\begin{props}\label{dif:ei}
  \(T_\Cc\) eliminates imaginaries to the level of projective spaces.
\end{props}
\begin{proof}
  Both the statement and the proof are analogous 
  to~\Cite[Proposition~4.2]{groupoids}.

  We need to show that any definable set \(S\) over parameters can be defined 
  with a canonical parameter. Since, by assumption, no new structure is 
  induced on \(\LL\), and any set is internal to \(\LL\), every such set is 
  Kolchin constructible. By Noetherian induction, it is enough to consider 
  \(S\) Kolchin closed.

  We claim that the algebra of differential polynomials on a sort \(U\) is 
  ind-definable in \(T_\Cc\). Indeed, it is precisely given by the 
  interpretation of scheme structure associated with \(U\) 
  (Definition~\ref{dif:scheme}). We only mention the definition of the 
  evaluation map (using the notation there): it is enough to the evaluation 
  on (the interpretation of) \(\oD\), since the rest is as in the algebraic 
  case.  We define the evaluation \(e_n:\Co{U}_n\x{}U\ra\LL\) by induction on 
  \(n\): \(e_0\) is the usual evaluation. If \(u\in{}U\), the map 
  \(d\mt{}e_n(d,u)'\) is a derivation on \(\Co{U}_n\), and so defines a 
  linear map from \(\d(\Co{U}_n)\) to \(\LL\). Inspection of the definition 
  of \(\Co{U}_{n+1}\) (for vector spaces) reveals that this map descends to a 
  linear map \(e_{n+1}(-,u)\) on \(\Co{U}_{n+1}\).

  The rest of the proof is the same as in~\Cite{groupoids}, namely, the 
  Kolchin closed set \(S\) is determined by the finite dimensional linear 
  space spanned by the defining equations, and this space is an elements of 
  some Grassmanian, which is, in turn, a closed subset of some projective 
  space.
\end{proof}

\printbibliography[maxnames=10]

\end{document}